\documentclass[12pt,leqno,fleqn]{amsart} 
\usepackage[utf8]{inputenc}

\usepackage[fleqn]{mathtools}

\usepackage{amstext,amsxtra,amssymb}
\usepackage{hyperref}
\usepackage{cite}
\usepackage{amsfonts}
\usepackage{pgfplots}
\usepackage{comment}

\usepackage{mathrsfs}

\newtheorem{theorem}{Theorem}

\newtheorem{lemma}[theorem]{Lemma}
 \newtheorem{prop}[theorem]{Proposition}
 \newtheorem{corollary}[theorem]{Corollary}

\usepackage[top=1.5in, bottom=1.5in, left=1.25in, right=1.25in]	{geometry}

\usepackage{tikz}
\usetikzlibrary{arrows}

\newcommand{\bc}{\begin{center}}
\newcommand{\ec}{\end{center}}
\newcommand{\beqn}{\begin{align*}}
\newcommand{\eeqn}{\end{align*}}
\newcommand{\benu}{\begin{enumerate}}
\newcommand{\eenu}{\end{enumerate}}
\newcommand{\bit}{\begin{itemize}}
\newcommand{\eit}{\end{itemize}}

\newcommand{\cS }{\mathcal{ S }}

\newcommand{\pdd}[3]{\ifx#2#3\frac{\partial^2 #1}{\partial #2^2}\else \frac{\partial^2 #1}{\partial #2\,\partial #3}\fi}

\newcommand{\ave}[1]{\left\langle #1 \right\rangle}

\newcommand{\ti}{\;\;\makebox[0pt]{$\top$}\makebox[0pt]{$\pitchfork$}\;\;}

 \author{Robert Kesler}    

\address{1217 21st Street, Santa Monica CA 90404, USA}
\email {robertmkesler@gmail.com}

\begin{document}

 \title[Improving Properties and Sparse Bounds for Discrete Spherical Means]{$\ell^p(\mathbb{Z}^d)$-Improving Properties and 
 Sparse Bounds for Discrete Spherical Maximal Averages}

\begin{abstract} We exhibit a range of $\ell ^{p}(\mathbb{Z}^d)$-improving properties for the discrete spherical maximal average in every dimension $d\geq 5$. The strategy used to show these improving properties is then adapted to establish sparse bounds, which extend the discrete maximal theorem of Magyar, Stein, and Wainger to weighted spaces. In particular, the sparse bounds imply that the discrete spherical maximal average is a bounded map from $\ell^2(w)$ into $\ell^2(w)$ provided $w^{\frac{d}{d-4}+\delta}$ belongs to the Muckenhoupt class $A_2$ for some $\delta>0.$
 \end{abstract}
  \maketitle
  \tableofcontents
 \section{Introduction}
Let $\textbf{A}^d_\lambda$ denote the continuous spherical averaging operator on $\mathbb{R}^d$ at radius $\lambda$, i.e.
\begin{align*}
\textbf{A}^d_\lambda f (x) = \int_{S^{d-1}} f(x-\lambda y) d \sigma (y),
\end{align*}
where $d \geq 2$, $S^{d-1}$ denotes the unit $d-1$ dimensional sphere in $\mathbb{R}^d$ and $\sigma$ is the unit surface measure on $S^{d-1}$. Stein establishes the spherical maximal theorem for $d \geq 3$ in \cite{Stein}, which states that $|| \sup_{\lambda} |\textbf{A}^d_\lambda |: L^p(\mathbb{R}^d) \rightarrow L^p(\mathbb{R}^d)|| <\infty$ for all $ \frac{d}{d-1} < p \leq \infty $. Bourgain examines the $d=2$ case in \cite{Bourgain} and shows $|| \sup_\lambda |\textbf{A}^2_\lambda| :L^p(\mathbb{R}^2) \rightarrow L^p(\mathbb{R}^2) || <\infty$ for all $2 < p \leq \infty$. 
The sharp $L^p(\mathbb{R}^d)$-$L^q(\mathbb{R}^d)$ result for $\sup_{1 \leq \lambda <2} |\mathbf{A}_\lambda^d|$ shown by Schlag in \cite{Schlag1} is as follows.

\begin{theorem}\label{Thm:0}
Let $d \geq 2$. Define $\mathcal{T}(d)$ to be the interior convex hull of $\{T_{d,j}\}_{j=1}^4$, where  
\begin{eqnarray*}
&T_{d,1} = (0,1) \qquad 
&T_{d,2} = \left(\frac{d-1}{d}, \frac{1}{d} \right) \\ 
&T_{d,3} = \left(\frac{d-1}{d}, \frac{d-1}{d} \right) \qquad  
 &T_{d,4} = \left( \frac{ d^2 - d}{d^2+1} ,  \frac{d^2 -d+2}{d^2+1} \right).
\end{eqnarray*}
Then for all $(\frac{1}{p}, \frac{1}{r}) \in \mathcal{T}(d)$ there exists a constant $A=A(d,p,r)$ such that
\begin{align*}
\left| \left| \sup_{1 \leq \lambda <2} | \textbf{A}^d_\lambda | \right| \right|_{L^p \to L^{r^\prime}}  \leq A.
\end{align*}
By rescaling, we obtain for all $(\frac{1}{p}, \frac{1}{r}) \in \mathcal{T}(d)$ and $\Lambda \in 2^{\mathbb{Z}}$, 
\begin{align*}
\left| \left| \sup_{\Lambda \leq \lambda < 2 \Lambda} |\textbf{A}^d_\lambda  | \right| \right|_{L^p \to L^{r^\prime}} \leq A \Lambda^{d(1/r^\prime - 1/p)}. 
\end{align*}
\end{theorem}
Lacey obtains a sparse extension of the continuous spherical maximal theorem in \cite{Lacey}. To state his result properly, we first need to set some notation for sparse bounds. Recall that a collection of cubes $\cS$ in $\mathbb{R}^d$ is called $\rho$-sparse if for each $Q \in \cS$, there is a subset
$E_Q \subset Q$ such that (a) $|E_Q| > \rho |Q|$, and (b) $\|  \sum_{Q \in \cS} 1_{E_Q}  \|_{L^\infty(\mathbb{R}^d)}
\leq \rho^{-1}$.  For a sparse collection $\cS$, a sparse bilinear $(p,r)$-form $\Lambda$ is defined by
\begin{align*}
\Lambda_{\cS,p,r}(f,g) : = \sum_{Q \in \cS} \ave{f}_{Q,p} \ave{g}_{Q,r} |Q|
\end{align*}
where $\ave{h}_{Q,t}:= \left(\frac{1}{|Q|} \sum_{x \in Q} |f(x)|^t \right)^{1/t}$ for any $t:1 \leq t < \infty$, cube $Q \subset \mathbb{Z}^d$, and $h: \mathbb{Z}^d \to \mathbb{C}$. 
Each $\rho$-sparse collection $\cS$ can be split into $O(\rho^{-2})$ many $\frac{1}{2}$-sparse collections. As long as $\rho^{-1}=O(1)$, its exact value is not relevant.
To simplify some of the arguments, we use the following definition introduced in \cite{Culiuc}:  for an operator $T$
acting on measurable, bounded, and compactly supported functions $f:\mathbb{R}^n \rightarrow \mathbb{C}$ and $1 \leq p, r < \infty$, define its sparse norm $\| T: (p,r) \|$ to be the infimum over all $C > 0$ such that for all pairs of measurable, bounded and compactly supported functions $f,g: \mathbb{R}^n \rightarrow \mathbb{C}$ 
\begin{align*}
|\langle Tf, g \rangle | \leq C \sup_\cS \Lambda_{\cS, p,r} (f,g) 
\end{align*}
where the supremum is taken over all $\frac{1}{2}$-sparse forms. A collection $\mathcal{C}$ of ``cubes" in $\mathbb{Z}^d$ is $\rho$-sparse provided there is a collection $\mathcal{S}$ of $\rho$-sparse cubes in $\mathbb{R}^d$ with the property that $\{R \cap \mathbb{Z}^d : R \in \mathcal{S}\}=\mathcal{C}.$ For a discrete operator $T$, define the sparse norm $||T: (p,r)||$ to be the infimum over all $C>0$ such that for all pairs of bounded and finitely supported functions $f, g : \mathbb{Z}^d \rightarrow \mathbb{C}$ 
\begin{align*}
| \langle T f, g \rangle| \leq C \sup_{\mathcal{S}} \Lambda_{\mathcal{S}, p,r} (f,g)
\end{align*}
where the supremum is taken over all $\frac{1}{2}$-sparse collections $\mathcal{S}$ consisting of discrete ``cubes." 
The sparse bounds obtained for continuous spherical maximal averages by Lacey in \cite{Lacey} are given by

\begin{theorem}\label{Thm:Lacey}
Let $d \geq 2$ and $\mathcal{R}_T(d)$ be as in Theorem \ref{Thm:0}. Then for all $(\frac{1}{p}, \frac{1}{r}) \in \mathcal{R}_T(d)$
\begin{align*}
\left| \left| \sup_{\lambda >0} |\textbf{A}^d_\lambda |~ : (p, r) \right| \right| <\infty. 
\end{align*}

\end{theorem}
Magyar, Stein, and Wainger prove their discrete spherical maximal theorem in \cite{Magyar}:  

\begin{theorem}\label{Thm:Magyar}
For each $\lambda \in \tilde{\Lambda}:= \left\{ \lambda >0 : \lambda^2 \in \mathbb{N} \right\}$ define the discrete spherical average 
\begin{align*}
\mathscr{A}_\lambda f(x)= \frac{1}{|\{ |y|=\lambda\}|} \sum_{y \in \mathbb{Z}^d : |y| = \lambda} f(x-y). 
\end{align*}
 Then for all $d \geq 5$ and $ \frac{d}{d-2}< p \leq \infty$
\begin{align*}
\left| \left| \sup_{\lambda \in \tilde{\Lambda}} |\mathscr{A}_\lambda | \right| \right|_{\ell^p \to \ell^p} <\infty. 
\end{align*}
\end{theorem}

\begin{figure}\label{f:}

\begin{tikzpicture}[scale=4] 
\draw[thick,->] (-.2,0) -- (1.2,0) node[below] {$ \frac 1 p$};
\draw[thick,->] (0,-.2) -- (0,1.2) node[left] {$ \frac 1 r$};
\draw[fill=green] (0,1) --  (.9,.1) 
-- (.9,.9) 
-- (0,1) 
;
\draw[fill=teal] (0,1) -- (.9,.95) 
-- (.9,.9) 
-- (0,1)

;
\draw (.9,.05) -- (.9,-.05) node[below] {$ \tfrac{d-2}{d}$};
\draw (.05,.1) -- (-.05,.1) node[left] {$ \tfrac {2}{d}$};
\draw (.05, .9) -- (-.05, .9) node[left] {$ \tfrac {d-2}{d}$};
\draw[loosely dashed] (0,1) -- (1.,1.) node[above] {$ (1,1)$} -- (1.,0); 
\draw (.6,.7) node {$ \mathcal{R}(d)$}; 
\end{tikzpicture}

\begin{tikzpicture}[scale=4] 
\draw[thick,->] (-.2,0) -- (1.2,0) node[below] {$ \frac 1 p$};
\draw[thick,->] (0,-.2) -- (0,1.2) node[left] {$ \frac 1 r$};
\draw[fill=yellow] (.1,.9) --  (.9,.1) 
-- (.9,.9) 
-- (.1,.9) 
;
\draw (.1,.05) -- (.1,-.05) node[below] {$ \tfrac {2}{d}$};
\draw (.9,.05) -- (.9,-.05) node[below] {$ \tfrac{d-2}{d}$};
\draw (.05,.1) -- (-.05,.1) node[left] {$ \tfrac {2}{d}$};
\draw (.05, .9) -- (-.05, .9) node[left] {$ \tfrac {d-2}{d}$};
\draw[loosely dashed] (0,1) -- (1.,1.) node[above] {$ (1,1)$} -- (1.,0); 
 \draw[domain=.66:.9,variable=\x,black] 
 plot({\x},{1-\x/(20(1-\x))});
\draw (.6,.7) node {$\mathcal{S}(d)$}; 
\end{tikzpicture}

\caption{The green region $\mathcal{R}(d)$ represents the range of uniform improving properties for $\sup_{\Lambda \leq \lambda < 2 \Lambda}\left| \mathscr{A}_\lambda \right|$ and sparse bounds for $\sup_{\lambda \in \tilde{\Lambda}} |\mathscr{A}_\lambda|$ that we are able to prove. The teal region adjacent to $\mathcal{R}(d)$ represents the range of improving properties for $\sup_{\Lambda \leq \lambda < 2 \Lambda}\left| \mathscr{A}_\lambda \right|$ and sparse bounds for $\sup_{\lambda \in \tilde{\Lambda}} |\mathscr{A}_\lambda|$ that we cannot prove or disprove. The yellow region $\mathcal{S}(d)$ represents the range of improving properties for $\sup_{\Lambda \leq \lambda < 2 \Lambda}\left| \mathscr{C}_\lambda \right|$ and $\sup_{\Lambda \leq \lambda < 2 \Lambda}\left| \mathscr{R}_\lambda \right|$ as well as sparse bounds for $\sup_{\lambda \in \tilde{\Lambda}}\left| \mathscr{C}_\lambda \right|$ and $\sup_{\lambda \in \tilde{\Lambda}}\left| \mathscr{R}_\lambda \right|$ that we are able to prove, where $\mathscr{C}_\lambda$ is given by \eqref{Def:c(lambda)} and $\mathscr{R}_\lambda=\mathscr{A}_\lambda - \mathscr{C}_\lambda$ is the residual term.} 
\end{figure}
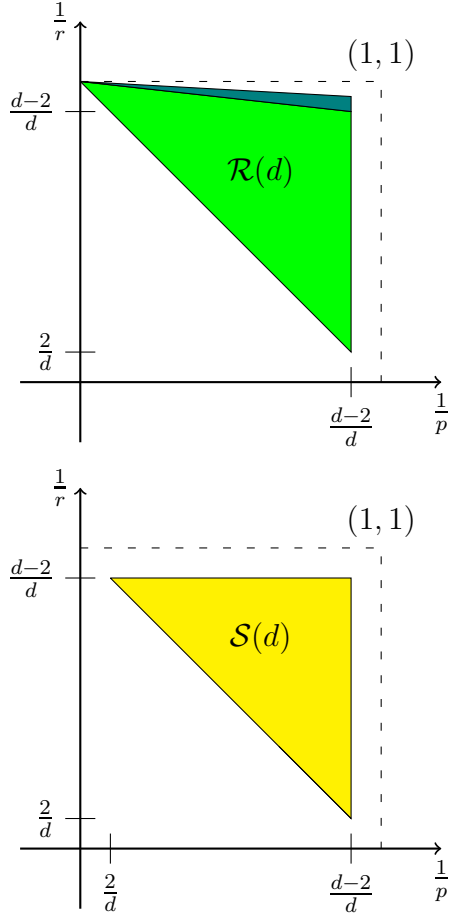
Our first theorem establishes a discrete analogue of Theorem \ref{Thm:0}: 
\begin{theorem}\label{Thm:1}
Let $d \geq 5$. Define $\mathcal{R}(d)$ to be the interior convex hull of 
\begin{align*}
R_{d,1}= \left(0,1 \right) \qquad 
R_{d,2}= \left(\frac{d-2}{d}, \frac{2}{d} \right) \qquad 
R_{d,3} = \left(\frac{d-2}{d}, \frac{d-2}{d} \right).
\end{align*} 
Then for all $(\frac{1}{p}, \frac{1}{r}) \in \mathcal{R}(d)$ there exists $A=A(d,p,r)$ such that for every $f \in \ell^p(\mathbb{Z}^d)$ and $\Lambda \in 2^{\mathbb{N}}$ 
\begin{align}\label{MainEst:Imp}
 \left| \left| \sup_{\Lambda \leq \lambda < 2 \Lambda  } |\mathscr{A}_\lambda |  \right| \right| _{\ell^p \to \ell^{r^\prime}} \leq A \Lambda^{d(1/r^\prime - 1/p)}.  
\end{align}
A necessary condition for \eqref{MainEst:Imp} to hold for all $\Lambda \in 2^{\mathbb{N}}$ is $\max \left\{ \frac{1}{p} +\frac{2}{d} , \frac{1}{r} + \frac{2}{pd}  \right\} \leq 1$. 
\end{theorem}
Our second theorem establishes the following discrete analogue of Theorem \ref{Thm:Lacey}:

\begin{theorem}\label{Thm:2}
Let $d \geq 5$ and $\mathcal{R}(d)$ be as in Theorem \ref{Thm:1}. Then for all $(\frac{1}{p}, \frac{1}{r}) \in \mathcal{R}(d)$
\begin{align}\label{Est:Sp}
\left| \left| \sup_{\lambda \in \tilde{\Lambda}} \left| \mathscr{A}_\lambda \right|: (p,r)  \right| \right| <\infty. 
\end{align}
A necessary condition for \eqref{Est:Sp} to hold is $\max\left\{ \frac{1}{p} + \frac{2}{d}, \frac{1}{r} +\frac{2}{pd} \right\} \leq 1$.

\end{theorem}

 \section{Discussion of Results}
 While the study of improving properties for discrete maximal averages is new, much effort has focused on obtaining $\ell^p(\mathbb{Z}^d)$-estimates for discrete operators in harmonic analysis since the foundational work of Bourgain on ergodic theorems concerning polynomial averages. For instance, a number of delicate $\ell^p(\mathbb{Z}^d)$-bounds are obtained in the setting of radon transforms in \cite{Pierce,Ionescu,Mirek}, fractional variants in \cite{Stein3,Pierce2}, and Carleson operators in \cite{Krause}. A well-known technique in this setting is the circle method of Hardy, Littlewood, and Ramanujan, which Magyar, Stein, and Wainger apply for the discrete spherical maximal averages in \cite{Magyar} to prove Theorem \ref{Thm:Magyar} by decomposing $\mathscr{A}_\lambda = \mathscr{C}_\lambda+ \mathscr{R}_\lambda $, where $\mathscr{C}_\lambda$ is consists of a sum of modulated and fourier-localized copies of the continuous spherical averaging operator and $\mathscr{R}_\lambda$ is the residual term. We shall define $\mathscr{C}_\lambda$, and thereby define $\mathscr{R}_\lambda$, in \S{2}.

In the case where the supremum is taken only over discrete spherical averages with radii belonging to a thin set, for example a lacunary sequence, one can expand the range of sparse and $\ell^p$-$\ell^q$ improving estimate beyond $\mathcal{R}(d)$ by using Kloosterman and Ramanujan sum refinements, and a good $L^\infty(\mathbb{T}^d)$ estimate on the symbol of $\mathscr{R}_\lambda$, namely $O_\delta( \lambda ^{- \frac{d-3}{2}+\delta})$ for all $\delta>0$ from \cite{Magyar1}. However, if the radii appearing in the supremum cluster too closely together, then one cannot reduce the argument to an estimate that is uniform in $\lambda$. It is for this reason that our analysis of the residual term $\mathscr{R}_\lambda$ in this paper is substantially more involved than in the lacunary case \cite{Kesler10}. Moreover, as this paper only considers the full set of radii, Kloosterman and Ramanujan sums along with a good $L^\infty(\mathbb{T}^d)$ bound on the symbol of the residual operator $\mathscr{R}_\lambda$ are not able to improve our results and are therefore omitted from the analysis. 

More than half of the paper is dedicated to obtaining sparse bounds for discrete maximal spherical averages in the full supremum case. Pointwise sparse domination for Calder\'on-Zygmund operators is obtained by Conde-Alonso and Rey in \cite{Conde-Alonso} and is recently obtained as a consequence of work by Lacey in \cite{Lacey2} on martingale transforms using a stopping time argument. Sparse form domination is a relaxation of the pointwise approach and holds in many settings, including Bochner-Riesz operators in \cite{ Kesler2} and oscillatory integrals in \cite{Lacey4} to name but a few.  

Recent work of Lacey establishes sparse form domination for the continuous spherical maximal averages using the improving estimates in Theorem \ref{Thm:0} and thereby shows a variety of weighted inequalities.
The underlying method of proof relies on Theorem \ref{Thm:0}, a certain continuity property derived by interpolating against a favorable $\ell^2(\mathbb{Z}^d) \rightarrow \ell^2(\mathbb{Z}^d)$ estimate, and a carefully applied Calder\'on-Zygmund decomposition in a manner related to Christ and Stein's analysis in \cite{Christ}. Moreover, there are several recent sparse results in the discrete setting involving random Carleson operators in \cite{Krause2}, the cubic Hilbert transform in \cite{Culiuc}, and a family of quadratically modulated Hilbert transforms in \cite{Kesler}.

The proof of Theorem \ref{Thm:1} reduces to showing that for all $(\frac{1}{p}, \frac{1}{r}) \in \mathcal{S}(d)$ there exists $A=A(d,p,r)$ such that for every $f \in \ell^p(\mathbb{Z}^d)$ and $\Lambda \in 2^{\mathbb{N}}$ 
\begin{align}
\left|\left| \sup_{\Lambda \leq \lambda < 2 \Lambda} |\mathscr{C}_\lambda  | : \right|\right|_{\ell^p \to \ell^{r^\prime}}  \leq A \Lambda^{d(1/r^\prime - 1/p)} \label{Est:Improving1} \\ 
 \left| \left| \sup_{\Lambda \leq \lambda < 2\Lambda} |\mathscr{R}_\lambda  | \right| \right|_{\ell^p \to \ell^{r^\prime}} \leq A \Lambda^{d(1/r^\prime - 1/p)} \label{Est:Improving2},
 \end{align}
where $\mathcal{S}(d)$ is the interior convex hull of 
\begin{align}
S_{d,1} = \left( \frac{2}d , \frac{d-2}d \right) \qquad  
S_{d,2} = \left(\frac{d-2}{d}, \frac{2}{d} \right) \qquad 
S_{d,3} =\left(\frac{d-2}{d}, \frac{d-2}{d} \right).
\end{align}
Indeed, estimate \eqref{MainEst:Imp} is an immediate consequence of interpolating estimates close to $(\frac{d-2}{d}, \frac{d-2}d)$ with the trivial endpoint estimate at $(0,1)$. Furthermore, the arguments for \eqref{Est:Improving1} and \eqref{Est:Improving2} rely on interpolating between favorable $\ell^2\rightarrow \ell^2$ bounds and boundary estimates arising from point-wise control of various kernels. See figure \ref{f:} for a depiction of $\mathcal{S}(d)$. 

The proof of Theorem \ref{Thm:2} is reduced to showing that for all $(\frac{1}{p}, \frac{1}{r}) \in \mathcal{S}(d)$ 
\begin{align}
& \left| \left| \sup_{\lambda \in \tilde{\Lambda}} \left| \mathscr{C}_\lambda \right|: (p,r)  \right| \right| <\infty\label{Est:MS}\\ 
& \left| \left| \sup_{\lambda \in \tilde{\Lambda}} \left| \mathscr{R}_\lambda \right|: (p,r)  \right| \right| <\infty\label{Est:ES}
\end{align}
in conjunction with a restricted weak-type type interpolation argument from \cite{Kesler10}.  
The arguments for \eqref{Est:MS} and \eqref{Est:ES} rely on the improving properties in \eqref{Est:Improving2} and \eqref{Est:Improving1} respectively.  

Once we obtain sparse bounds for $\sup_{\lambda \in \tilde{\Lambda}} |\mathscr{C}_\lambda|$ and $\sup_{\lambda \in \tilde{\Lambda}} |\mathscr{R}_\lambda|$ throughout $\mathcal{S}(d)$, we extend these estimates to $(p,r)$-sparse bounds for $(\frac{1}p, \frac{1}{r})$ arbitrarily close to $(0,1)$ by reducing the problem to obtaining restricted weak-type sparse bounds via Theorem \ref{Restr-Sparse} and then applying a localized variant of Theorem \ref{Thm:2} near $(\frac{d-2}d, \frac{d-2}2)$ as described in Theorem \ref{Cor:Loc}. 

A weighted consequence of the sparse bounds in Theorem \ref{Thm:2} is 

\begin{corollary}
For all $d \geq 5$, $w: \mathbb{Z}^d \rightarrow [0, \infty)$, and $\delta>0$ such that $w^{\frac{d}{d-4}+\delta}  \in A_2$,
\begin{align}\label{Cor:Est}
\left| \left| \sup_{\lambda \in \tilde{\Lambda}} |\mathscr{A}_\lambda | : \ell^2(w) \rightarrow \ell^2(w) \right|\right| <\infty. 
\end{align}
Moreover, \eqref{Cor:Est} holds for all weights $w$ in the intersection of the Muckenhoupt class $A_2$ and the reverse H\"older class $RH_r$, since we may choose $r=r(\delta)<2$ so that $w \in A_2 \cap RH_r$ guarantees $w^{\frac{d}{d-4}+\delta} \in A_2$. 
\end{corollary}

To the author's knowledge, no $\ell^p(\mathbb{Z}^d)$-improving properties, sparse bounds, or weighted inequalities were previously known for the discrete spherical maximal averages in the full supremum case. We leave open the question of whether the ranges for $\ell^p(\mathbb{Z}^d)$-improving properties in Theorem \ref{Thm:1} and sparse bounds in Theorem \ref{Thm:2} are sharp. 

\vspace{5mm}
This paper is structured as follows: \S{3} introduces relevant background from the proof of the discrete spherical maximal theorem in \cite{Magyar}, \S{4} contains the proof of estimate \eqref{Est:Improving1}, \S{5} contains the proof of estimate \eqref{Est:Improving2}, \S{6} contains the proof of estimate \eqref{Est:MS} , \S{7} contains the proof of estimate \eqref{Est:ES}, \S{8} contains the proof of estimate \eqref{Est:Sp}, and \S{9} contains the counterexamples for the negative content of Theorems \ref{Thm:1} and \ref{Thm:2}.   

\vspace{5mm}

The letter $A$ is always used in the mathematical expressions of this paper to denote a positive constant, which depends only on inessential parameters and whose precise value is allowed to change from line to line.

\section{Decomposition and Transference of Discrete Spherical Averages}

We now introduce the decomposition of the discrete spherical average $\mathscr{A}_\lambda = \mathscr{C}_\lambda + \mathscr{R}_\lambda$ and a transference lemma, both from \cite{Magyar}. The symbol of the multiplier $\mathscr{A}_\lambda$ for $\Lambda \leq \lambda < 2 \Lambda$ and $\Lambda \in 2^{\mathbb{N}}$ can be written as
\begin{align}\label{Def:a(lambda)}
a_\lambda (\xi) =& \sum_{q=1}^\Lambda \sum_{a \in \mathbb{Z}^\times_q } a_\lambda^{a/q}(\xi)
\end{align}
where 
\begin{align}
a_\lambda^{a/q}(\xi) =& e^{-2 \pi i \lambda^2 a/q} \sum_{\ell \in \mathbb{Z}^d} G(a/q, \ell) J_\lambda (a/q, \xi - \ell /q)\\
 G(a/q, \ell)=& \frac{1}{q^d} \sum_{n \in \mathbb{Z}^d/q \mathbb{Z}^d} e^{ 2\pi i |n|^2 a/q} e^{-2 \pi i n \cdot l /q}\label{Def:GS} \\ J_\lambda (a/q , \xi) =& \frac{ e^{ 2 \pi}}{\lambda^{d-2}} \int_{I(a,q)}e^{-2 \pi i \lambda^2 \tau}  \frac{e^{ \frac {- \pi |\xi|^2}{2 ( \epsilon - i \tau)} } }{ 2 (\epsilon - i \tau))^{d/2}} d \tau \\  \epsilon =& \frac{1}{\lambda^2}  
\end{align}
and 
$I(a,q)=\left[- \frac{\beta}{q \Lambda}, \frac{\alpha}{q \Lambda} \right]$ for $\alpha= \alpha(a/q, \Lambda) \simeq 1, \beta = \beta(a/q, \Lambda) \simeq 1$.
Next, we shall pick $\Phi \in \mathscr{C}^\infty ([-1/4, 1/4]^d)$ such that $\Phi\equiv 1$ on $[-1/8, 1/8]^d$ and for $q \in \mathbb{N}$, set $\Phi_q(\xi) = \frac{1}{q^{d}}\Phi\left(\frac{\xi}{q}\right)$, and define
\begin{align}
b_\lambda (\xi) =& \sum_{q=1}^\Lambda \sum_{a \in \mathbb{Z}^\times_q} b_\lambda^{a/q} (\xi) \label{Def:b(lambda)} \\ 
b_\lambda^{a/q} (\xi) =& e^{-2 \pi i \lambda^2 a/q} \sum_{\ell \in \mathbb{Z}^d/q \mathbb{Z}^d} G(a/q, \ell) \Phi_q(\xi - \ell /q)J_\lambda (a/q, \xi - \ell /q) \nonumber 
\end{align}
along with $\mathscr{B}^{a/q}_\lambda: f \mapsto f * \check{b}_\lambda^{a/q}$ and $\mathscr{B}_\lambda: f \mapsto f * \check{b}_\lambda$. 
Therefore, $b_\lambda^{a/q}$ is constructed from $a_\lambda^{a/q}$
by inserting cutoff factors into each summand of $a_\lambda^{a/q}$ at length scale $\frac{1}{q}$. We subsume the difference $b_\lambda - a_\lambda$ into the residual term $\mathscr{R}_\lambda$. Lastly, it is convenient to extend the domain of integration in the definition of $J_\lambda$ to all of $\mathbb{R}$ and subsume this difference as part of the residual term $\mathscr{R}_\lambda$. To this end, we introduce
\begin{align*}
I_\lambda (\xi) = \frac{ e^{ 2 \pi}}{\lambda^{d-2}} \int_{-\infty}^\infty e^{-2 \pi i \lambda^2 \tau} \frac{e^{ \frac {- \pi |\xi|^2}{2 ( \epsilon - i \tau)} }} {(2 (\epsilon - i \tau))^{d/2} } d \tau 
\end{align*}
and let
\begin{align}
c_\lambda (\xi) =& \sum_{q=1}^\Lambda \sum_{a \in \mathbb{Z}^\times_q} c_\lambda^{a/q} (\xi) \label{Def:c(lambda)} \\ 
c_\lambda^{a/q} (\xi) =&  e^{-2 \pi i \lambda^2 a/q} \sum_{\ell \in \mathbb{Z}^d/q\mathbb{Z}^d} G(a/q, \ell) \Phi_q(\xi - \ell /q)I_\lambda ( \xi - \ell /q) \nonumber
\end{align}
along with $\mathscr{C}^{a/q}_\lambda: f \mapsto f * \check{c}_\lambda^{a/q}$, and $\mathscr{C}_\lambda : f \mapsto f* \check{c}_\lambda$.
Since $I_\lambda  = c_d \widehat{d \sigma_\lambda}$, where $c_d$ is a dimensional constant and $d \sigma_\lambda$ is the unit surface measure of the sphere in $\mathbb{R}^d$ of radius $\lambda$, 
\begin{align}
c_\lambda^{a/q}(\xi) =c_d e^{-2 \pi i \lambda^2 a/q} \sum_{\ell \in \mathbb{Z}^d/q \mathbb{Z}^d} G(a/q, \ell) \Phi_q(\xi - \ell /q) \widehat{d \sigma_\lambda} (\xi- \ell /q). 
\end{align}
It follows that
\begin{align*} 
c_\lambda (\xi) =&  \sum_{q =1}^\Lambda \sum_{a \in \mathbb{Z}^\times_q} c_\lambda^{a/q}(\xi) \\=& c_d  \sum_{q =1}^\Lambda \sum_{a \in \mathbb{Z}^\times_q} e^{-2 \pi i \lambda^2 a/q} \sum_{\ell \in \mathbb{Z}^d} G(a/q, \ell) \Phi_q(\xi - \ell /q)\widehat{d\sigma_\lambda} ( \xi - \ell /q),
\end{align*}
$\mathscr{C}_\lambda : f \mapsto f * \check{c}_\lambda$, and $\mathscr{C}_\lambda^{a/q} : f \mapsto f * \check{c}_\lambda^{a/q}$. Lastly, for each $\lambda \in \tilde{\Lambda}$ let $\mathscr{R}_\lambda  = \mathscr{A}_\lambda - \mathscr{C}_\lambda$. An important fact is the Gauss sum estimate 
\begin{align}\label{Est:GS}
\left| G(a/q, \ell) \right| \leq A q^{-d/2}
\end{align}
which holds uniformly in $a,q,$ and $\ell$; this is well-known in the case of $d=1$ case from which the $d \geq 2$ case immediately follows.  We now recall two estimates:
\begin{align}
\left| \left| \sup_{\Lambda \leq \lambda < 2 \Lambda} \left| \mathscr{C}_\lambda^{a/q} \right| \right| \right|_{\ell^2 \to \ell^2}   \leq &A q^{-d/2}\qquad \forall a \in \mathbb{Z}^\times_q, 1 \leq q \leq \Lambda, \Lambda \in 2^{\mathbb{N}} \label{Est:Cl2}\\ 
\left| \left| \sup_{\Lambda \leq \lambda < 2 \Lambda} \left| \mathscr{R}_\lambda  \right| \right| \right|_{\ell^2 \to \ell^2} \leq& A \Lambda^{2-d/2}\qquad \forall a \in \mathbb{Z}^\times_q, 1 \leq q \leq \Lambda, \Lambda \in 2^{\mathbb{N}}  \label{Est:Rl2}
\end{align} 
which are Propositions $3.1$ and $4.1$ from \cite{Magyar} respectively. 
Naturally, these favorable $\ell^2$-bounds are related to the decay of the Gauss sum in \eqref{Est:GS}. 
Furthermore, from the fact that for each $d \geq 5$  there is $A=A(d)$ such that for all $\lambda \in \tilde{\Lambda}$  
\begin{align*}
\frac{ \lambda^{d-2} }{A} \leq \left|\left\{ x \in \mathbb{Z}^d: |x| = \lambda \right\} \right| \leq A \lambda^{d-2},
\end{align*}
the pointwise estimate $\left| \mathscr{A}_\lambda f (x)\right| \leq A \Lambda^2 \left[ \frac{1}{\Lambda^d} \sum_{|y| \leq \Lambda} 
|f(x-y)| \right]$ follows and so
\begin{align}
\left| \left| \sup_{\Lambda \leq \lambda < 2 \Lambda} |\mathscr{A}_\lambda| \right| \right|_{\ell^1(\mathbb{Z}^d)}  \leq A \Lambda^2 || f || _{\ell^1(\mathbb{Z}^d)}.
\end{align}
The transference lemma from \cite{Magyar} can be phrased as follows: 

\begin{lemma}\label{L:Transfer}
For $d \geq 1$ and an integer $q \geq 1$ suppose that $m:[-1/2, 1/2)^d \rightarrow B$ is supported on $[-1/(2q), 1/(2q))^d$, where $B$ is any Banach space. Set
\begin{align*}
m_{per}^q(\xi) = \sum_{\ell \in \mathbb{Z}^d} m(\xi- \ell/q)
\end{align*}
and $T^q_{dis}$ be the convolution operator on $\mathbb{Z}^d$ with $m_{per}^q$ as its multiplier, i.e. for all $f \in \ell^1(\mathbb{Z}^d)$
\begin{align*}
\widehat{T^q_{dis} f } (\xi)  = m_{per}^q(\xi)\hat{f}(\xi) \qquad \forall \xi \in [-1/2, 1/2)^d. 
\end{align*}
Moreover, let $T$ be the convolution operator on $\mathbb{R}^d$ with $m$ as its multiplier. Then there is a constant $A$ such that for any $1 \leq p < \infty$
\begin{align}
\left| \left| T^q_{dis} \right| \right|_{\ell^p(\mathbb{Z}^d) \rightarrow \ell^p_{B}(\mathbb{Z}^d)} \leq A \left| \left| T \right| \right|_{L^p(\mathbb{R}^d) \rightarrow L^p_{B}(\mathbb{R}^d)}.  
\end{align}
\end{lemma}
Before applying Lemma \ref{L:Transfer} in \S{4} to obtain the sparse bound \eqref{Est:MS}, we shall need to set more notation. First let $\{\psi_{2^k}\}_{k \in \mathbb{Z}}$ be a standard Littlewood-Paley decomposition where each $\psi_k$ is supported in $\{ 2^{k-1} \leq |\xi| \leq 2^{k+1}\}$. For all $q \in \mathbb{N}$ and $N, \Lambda \in 2^{\mathbb{N}}$ such that $N \leq \frac{\Lambda}{q}$ define $P^q_{N/\Lambda}$ for all $f \in \ell^1(\mathbb{Z}^d)$ according to
\begin{align*}
\widehat{ P^q_{N/\Lambda}f} (\xi) = \sum_{\ell \in \mathbb{Z}^d/q \mathbb{Z}^d} \psi_{N/\Lambda}(\xi-\ell/q) \hat{f}(\xi) \qquad \forall \xi \in [-1/2, 1/2)^d. 
\end{align*}
Moreover, for any $\# \in \mathbb{R}_{+}$ let $P_{\leq\#}$ be the operator defined by 
\begin{align*}
\widehat{ P^q_{\leq \#}(f)}(\xi) = \sum_{\ell \in \mathbb{Z}^d / q \mathbb{Z}^d} \sum_{2^k \leq \#} \psi_k(\xi-\ell/q) \widetilde{\Phi}_q(\xi- \ell/q) \hat{f}(\xi) \qquad \forall \xi \in [-1/2, 1/2)^d
\end{align*}
where $\widetilde{\Phi}_q(\cdot) := \tilde{\Phi}(q \cdot) $ for a fixed function $\tilde{\Phi} \in \mathscr{C}^\infty([-3/8, 3/8]^d)$ satisfying $\tilde{\Phi} \equiv 1$ on $[-1/4, 1/4]^d$. 
For convenience, we will just write $P_{N/\Lambda}$ and $P_{< \#}$ instead of $P_{N/\Lambda}^q$ and $P^d_{\leq \#}$; the dependence on $q$ will be implicit but nonetheless clear from context. 
\begin{lemma}
For every $d \geq 5, \delta>0, \frac{d}{d-2} < p \leq 2, N \in 2^{\mathbb{N}}, q \in \mathbb{N}$, and $a \in \mathbb{Z}^\times_q$, there exists $A=A(d,p,\delta)$ such that 
\begin{align}
&\left| \left| \sup_{ \Lambda \geq Nq} \sup_{\Lambda \leq \lambda < 2 \Lambda} \left| \mathscr{C}_{\lambda} ^{a/q} P_{N/ \Lambda}  \right|  \right| \right|_{\ell^p \to \ell^{p}} \leq A N^{1-d(1-1/p)+\delta}  q^{-d(1-1/p)+\delta}  \label{Est:Int} \\ &\left| \left| \sup_{ \Lambda \geq q} \sup_{\Lambda \leq \lambda < 2 \Lambda} \left| \mathscr{C}_{\lambda} ^{a/q} P_{\leq  \frac{1}{\Lambda}} \right| \right| \right| _{\ell^p \to \ell^p} \leq A  q^{-d(1-1/p)+\delta} \label{Est:Int1}.
\end{align}

\end{lemma}
We shall need \eqref{Est:Int1} for the proof of Theorem \ref{Thm:Sp-M}. 
\begin{proof}
Again choosing $\tilde{\Phi} \in \mathscr{C}^\infty([-3/8, 3/8]^d)$ satisfying $\tilde{\Phi} \equiv 1$ on $[-1/4, 1/4]^d$ yields
\begin{align}\label{Exp:Decomp}
 e^{2 \pi i \lambda^2 a/q} \cdot c_\lambda^{a/q}(\xi) =&  \left[ \sum_{\ell \in \mathbb{Z}^d} G(a/q, \ell) \tilde{\Phi}_q(\xi - \ell /q) \right] \cdot \left[ \sum_{\ell \in \mathbb{Z}^d}  \Phi_q(\xi - \ell /q) \widehat{d \sigma_\lambda} (\xi - \ell /q) \right] \\=:&  c^{a/q}_{\lambda,1} (\xi) \cdot c^{a/q}_{\lambda ,2}(\xi) \nonumber. 
\end{align}
Letting $T_m$ denote the convolution operator with corresponding symbol $m \in L^\infty(\mathbb{T}^d)$, 
we therefore obtain
\begin{align}\label{Est:MP-0}
& \left| \left| \sup_{  \Lambda \geq Nq} ~\sup_{\Lambda \leq \lambda < 2 \Lambda} \left| \mathscr{C}_{\lambda} ^{a/q} P_{N/ \Lambda}|\right| \right| \right|_{\ell^2 \to \ell^2} \leq
\left| \left| \sup_{  \Lambda \geq Nq} ~\sup_{\Lambda \leq \lambda < 2 \Lambda} \left| T_{c^{a/q}_{\lambda,2}} \right|  \right| \right| _{\ell^2 \ti \ell^2} \cdot  \left| \left| T_{c_{\lambda,1}^{a/q}} \right| \right|_{\ell^2 \to \ell^2} . 
\end{align}
An application of Lemma \ref{L:Transfer} to the family of symbols
\begin{align*}
c^q_{\lambda, N}(\xi) := \psi_{N/ \Lambda} (\xi) \Phi_q(\xi)  \widehat{d \sigma_\lambda}(\xi)
\end{align*} 
produces
\begin{align}\label{Est:MP-1}
& \left \lVert \sup_{  \Lambda \geq Nq} ~\sup_{\Lambda \leq \lambda < 2 \Lambda} \left| T_{c^{a/q}_{\lambda,2}} \right| \right \rVert_{\ell^2 \to \ell^2}   
\leq  A 
\left| \left| \sup_{  \Lambda \geq Nq} ~\sup_{\Lambda \leq \lambda < 2 \Lambda} \left| T_{c^q_{\lambda,N}} \right|  \right| \right|_{L^2 \to L^2}. 
\end{align}
By the Plancherel equality and Gauss sum estimate \eqref{Est:GS},  
\begin{align}\label{Est:l2-1}
\left| \left| T_{c_{\lambda, 1}^{a/q}}  \right| \right|_{\ell^2 \to \ell^2} \leq A  q^{-d/2}.
\end{align}
Moreover,
\begin{align}\label{Est:cont-max}
\left| \left| \sup_{  \Lambda \geq Nq} ~\sup_{\Lambda \leq \lambda < 2 \Lambda} \left| T_{c^q_{\lambda, N}} \right|  \right| \right|_{L^2 \to L^2} \leq A N^{1-d/2}. 
\end{align}
Indeed, for fixed $\lambda, \left| \left| T_{c^q_{\lambda, N}}  \right| \right|_{L^2 \to L^2} \leq A N^{1/2-d/2}$ on account of the decay of $\widehat{ d\sigma_\lambda }$ on the support of $\psi_{N/ \Lambda}$. The additional factor of $N^{1/2}$ appearing on the right side of \eqref{Est:cont-max} arises from the supremum over $\lambda$ and can be justified using standard techniques. See, for example, \cite{Stein4} for details. Combining  \eqref{Est:MP-0}, \eqref{Est:MP-1},\eqref{Est:l2-1}, and \eqref{Est:cont-max} yields
\begin{align}\label{Est:maxl2}
\left| \left| \sup_{  \Lambda \geq Nq} ~\sup_{\Lambda \leq \lambda < 2 \Lambda} \left| \mathscr{C}_{\lambda} ^{a/q} P_{N/ \Lambda}\right|   \right| \right| _{\ell^2 \to \ell^2} \leq A    N^{1-d/2} q^{-d/2}.  
\end{align}
Furthermore, from the estimate
\begin{align}\label{Est:Point}
|T_{c^q_{\lambda, N}} f(x)| \leq A  N M_{HL}f(x) \qquad \forall x \in \mathbb{Z}^d
\end{align}
it again follows from Lemma \ref{L:Transfer} that for every $\delta >0$
\begin{align}\label{Est:MP-3}
\left| \left| \sup_{  \Lambda \geq Nq} ~\sup_{\Lambda \leq \lambda < 2 \Lambda} \left| \mathscr{C}_{\lambda} ^{a/q} P_{N/ \Lambda}
\right|   \right| \right|_{\ell^{1+\delta} \to \ell^{1+\delta}} \leq A N  \left| \left| T_{c_{\lambda, 1}^{a/q}}  \right| \right|_{\ell^{1+\delta} \to \ell^{1+\delta}}. 
\end{align}
From the fact that
\begin{align}\label{Est:Gauss1}
\sum_{\ell \in \mathbb{Z}^d / q \mathbb{Z}^d} G(a/q,\ell) e ^{-2 \pi i  y \cdot \ell /q} = e^{2 \pi i |y|^2 a/q} \qquad \forall a \in \mathbb{Z}_q^{\times}, q \in \mathbb{N}, y \in \mathbb{Z}^d/q \mathbb{Z}^d,
\end{align}
we obtain
\begin{align*}
\left| \check{c}_{\lambda, 1}^{a/q}(x) \right| \leq A |\check{\tilde{\Phi}}_q (x) | \qquad \forall x \in \mathbb{Z}^d,
\end{align*}
so that for some $A = A(d)$
\begin{align}\label{Est:l1-1}
\left| \left| T_{c_{\lambda, 1}^{a/q}}  \right| \right|_{\ell^1 \to \ell^1} \leq A. 
\end{align}
Interpolating between \eqref{Est:l2-1} and \eqref{Est:l1-1} yields for every $1 \leq p \leq 2$ 
\begin{align}\label{Est:MP-2}
|| T_{c_{\lambda,1}^{a/q}}  ||_{\ell^p \to \ell^p} \leq A q^{-d(1-1/p)}.
\end{align}
Plugging \eqref{Est:MP-2} into \eqref{Est:MP-3} gives for every $0 <\delta \leq 1$
\begin{align}\label{Est:MP-4}
 \left| \left| \sup_{  \Lambda \geq Nq} ~\sup_{\Lambda \leq \lambda < 2 \Lambda} \left| \mathscr{C}_{\lambda} ^{a/q} P_{N/ \Lambda} |  \right| \right| \right|_{\ell^{1+\delta} \to \ell^{1+\delta}}
 \leq A N q^{- \frac{d\delta}{1+\delta}}.
\end{align}
Interpolating between \eqref{Est:maxl2} and \eqref{Est:MP-4} shows estimate \eqref{Est:Int}. Using 
\begin{align}\label{Est:maxl20}
\left| \left| \sup_{\Lambda}\sup_{\Lambda \leq \lambda < 2 \Lambda} |\mathscr{C}^{a/q}_\lambda P_{\leq 1/ \Lambda}  | \right| \right|_{\ell^2 \to \ell^2} \leq & A q^{-d/2} 
\end{align}
and the pointwise bound
\begin{align*}
\sup_{\Lambda}\sup_{\Lambda \leq \lambda < 2 \Lambda} | P_{\leq 1/ \Lambda}\left(f * \check{\Phi}_q * d \sigma_\lambda \right) (x)| \leq A M_{HL}f(x) \qquad \forall x \in \mathbb{Z}^d,
\end{align*}
estimate \eqref{Est:Int1} is similarly obtained, and so the details are omitted. 
\end{proof}

\section{Improving Properties for $\sup_{\Lambda \leq \lambda < 2 \Lambda} |\mathscr{C}_\lambda|$}
Our goal in this section is to obtain estimate \eqref{Est:Improving1}, which is the improving property for the $\mathscr{C}_\lambda$ term. 
The argument relies on interpolating between the $\ell^2\rightarrow \ell^2$ bound \eqref{Est:maxl2} and straightforward boundary estimates related to \eqref{Est:Point}. We begin with an elementary lemma, which will also be used later in showing estimates \eqref{Est:MS} and \eqref{Est:ES}.

\begin{lemma}\label{L:Bump}
For every $d \geq 1$ and $\Lambda \in 2^{\mathbb{N}}$, let $\phi^1_\Lambda : \mathbb{Z}^d \to \mathbb{R}$ be given by
\begin{align}\label{Def:Bump}
\phi^1_{\Lambda}(x):=& \frac{1}{\Lambda^d} \left[\frac{1}{1+\frac{|x|}{\Lambda}} \right]^{2d} \qquad \forall x \in \mathbb{Z}^d.
\end{align}
Then for every $1 \leq  p ,r  \leq \infty$ such that $\frac{1}{p} +\frac{1}{r} \geq 1$   
\begin{align}\label{Est:Aux}
|| f*\phi^1_\Lambda ||_{\ell^{r^\prime}(\mathbb{Z}^d)} \leq A \Lambda^{d(1/r^\prime - 1/p)} ||f||_{\ell^p(\mathbb{Z}^d)}. 
\end{align}
\end{lemma}
\begin{proof}
Estimate \eqref{Est:Aux} is trivial when $r^\prime = p$  as the kernel belongs to $\ell^1(\mathbb{Z}^d)$ uniformly in $\Lambda$. The estimate when $r^\prime = \infty$ follows immediately from H\"older's inequality. Interpolating between these two cases yields the conclusion of Lemma\ref{L:Bump}. 
\end{proof}
We now use Lemma \ref{L:Bump} to deduce the following improving property. 
\begin{lemma}\label{L:Est-MP-Imp}
Let $d \geq 5$ and $(\frac{1}{p}, \frac{1}{r}) \in \mathcal{S}(d)$. Then there exists $A=A(d,p,r)$ and $\eta=\eta(d,p,r) >0$ such that for every $\Lambda, N \in 2^{\mathbb{N}}$ such that $1 \leq N \leq \frac{\Lambda}{q} $
\begin{align}
\left| \left| \sup_{\Lambda \leq \lambda < 2 \Lambda} |\mathscr{C}^{a/q}_\lambda P_{N/ \Lambda}  |  \right| \right|_{\ell^p \to \ell^{r^\prime}} \leq & A N^{-\eta} q^{-2-\eta} \Lambda^{d(1/r^\prime - 1/p)}  \label{Est:M0}\\
\left| \left| \sup_{\Lambda \leq \lambda < 2 \Lambda} |\mathscr{C}^{a/q}_\lambda P_{\leq 1/\Lambda}  | \right| \right|_{\ell^p \to \ell^{r^\prime}}  \leq & A q^{-2-\eta} \Lambda^{d(1/r^\prime - 1/p)}  \label{Est:M1}.
\end{align}

\end{lemma}
\begin{proof}
The proof is by interpolation. Estimates \eqref{Est:maxl2} and \eqref{Est:maxl20} immediately yield
\begin{align}
\left| \left| \sup_{\Lambda \leq \lambda < 2 \Lambda} |\mathscr{C}^{a/q}_\lambda P_{N/ \Lambda} |  \right| \right|_{\ell^2 \to \ell^2} \leq & A N^{1-d/2} q^{-d/2} \label{Est:maxl21}\\
\left| \left| \sup_{\Lambda \leq \lambda < 2 \Lambda} |\mathscr{C}^{a/q}_\lambda P_{\leq 1/ \Lambda}  | \right| \right| _{\ell^2 \to \ell^2}\leq & A q^{-d/2} \label{Est:maxl22}.
\end{align}
We next invoke the estimates valid for all $M \in \mathbb{N}, 1 \leq N \leq \frac{\Lambda}{q},$ and $x \in \mathbb{Z}^d$
\begin{align}
\sup_{\Lambda \leq \lambda < 2 \Lambda} |\mathscr{C}^{a/q}_\lambda P_{N/ \Lambda} f(x)  |  \leq &  \frac{A_M }{\Lambda^d} N  |f|*\left[ \Lambda^d \phi^1_\Lambda(x) \right]^M \label{Est:Ker10} \\ \sup_{\Lambda \leq \lambda < 2 \Lambda} |\mathscr{C}^{a/q}_\lambda P_{\leq 1/ \Lambda} f(x) |  \leq &  \frac{A_M}{\Lambda^d} |f|* \left[ \Lambda^d \phi^1_\Lambda (x) \right]^{M}\label{Est:Ker11}
\end{align}
and Lemma \ref{L:Bump} to deduce that for all $\left( \frac{1}{p} , \frac{1}{r} \right) \in [0,1]^2$ such that $\max\{\frac{1}{p},\frac{1}{r}\} = 1$ 
\begin{align}
\left| \left| \sup_{\Lambda \leq \lambda < 2 \Lambda} |\mathscr{C}^{a/q}_\lambda P_{N/ \Lambda}  | \right| \right|_{\ell^p \to \ell^{r^\prime}} \leq & A N  \Lambda^{d(1/r^\prime - 1/p)} \label{Est:Boun1} \\ \left| \left| \sup_{\Lambda \leq \lambda < 2 \Lambda} |\mathscr{C}^{a/q}_\lambda P_{\leq 1/\Lambda}  | \right| \right|_{\ell^p \to \ell^{r^\prime}} \leq & A \Lambda^{d(1/r^\prime - 1/p)}. \label{Est:Boun2}
\end{align}
Interpolating \eqref{Est:maxl21} and \eqref{Est:Boun1} yields \eqref{Est:M0}, while interpolating \eqref{Est:maxl22} and \eqref{Est:Boun2} yields \eqref{Est:M1}.

\end{proof}
A direct consequence of Lemma \ref{L:Est-MP-Imp} is estimate \eqref{Est:Improving1}, which we record separately as
\begin{prop}\label{P:Imp-M}
Let $d \geq 5$ and $(\frac{1}{p}, \frac{1}{r}) \in \mathcal{S}(d)$. Then there exists $A=A(d,p,r)$ such that for all $\Lambda \in 2^{\mathbb{N}}$
\begin{align*}
\left| \left| \sup_{\Lambda \leq \lambda < 2 \Lambda} |\mathscr{C}_\lambda  |\right| \right|_{\ell^p \to \ell^{r^\prime}} \leq A \Lambda^{d(1/r^\prime - 1/p)}. 
\end{align*}

\end{prop}
\begin{proof}
The corollary follows by summing estimate \eqref{Est:M0} over all $N \in 2^{\mathbb{N}}: 1 \leq N \leq \frac{\Lambda}{q}, a \in \mathbb{Z}^\times_q$, and $1 \leq q \leq \Lambda$. 
\end{proof}

\section{Improving Properties for $\sup_{\Lambda \leq \lambda < 2 \Lambda} |\mathscr{R}_\lambda|$}
We obtain estimate \eqref{Est:Improving2} by showing improving properties for $\sup_{\Lambda \leq \lambda < 2 \Lambda} | \mathscr{A}_\lambda - \mathscr{B}_\lambda |$ and $\sup_{\Lambda \leq \lambda < 2 \Lambda} | \mathscr{B}_\lambda - \mathscr{C}_\lambda |$ separately. Recall that  $\mathscr{A}_\lambda: f \mapsto f * \check{a}_\lambda, \mathscr{B}_\lambda: g \mapsto g * \check{b}_\lambda, \mathscr{C}_\lambda: h \mapsto h * \check{c}_\lambda$, where the symbols $a_\lambda, b_\lambda,$ and $c_\lambda$ are defined in \eqref{Def:a(lambda)}, \eqref{Def:b(lambda)}, and \eqref{Def:c(lambda)} respectively. The following result is needed to obtain improving properties for $\sup_{\Lambda \leq \lambda < 2 \Lambda} | \mathscr{A}_\lambda - \mathscr{B}_\lambda |$. 
\begin{lemma}\label{L:Est-mu-Imp}
For $d \geq 5, q \in \mathbb{N}, a \in \mathbb{Z}_q^{\times}$ and $\tau \in \mathbb{R}$, let
 
\begin{align}\label{Def:mu}
\mu_{a/q,\tau,\lambda} (\xi)= \sum_{\ell \in \mathbb{Z}^d}G(a/q, \ell) (1- \Phi_q (\xi - \ell /q)) e^{- \pi |\xi - \ell/q|^2/2 ( \epsilon - i \tau)}.
\end{align}
Then for all $(\frac{1}{p}, \frac{1}{r}) \in \mathcal{S}(d)$, there exists $A=A(d,p,r)$ and $\eta=\eta(d,p,r)>0$ such that for all $k \in \mathbb{Z}_+ , \Lambda \in 2^{\mathbb{N}}$, and $\tau \in I_k(\Lambda) :=\left\{ \tau \in \mathbb{R}: \frac{2^k-1}{\Lambda^2} \leq |\tau| \leq  \frac{2^k}{\Lambda^2}\right\}$
\begin{align}
\left \lVert \sup_{\Lambda \leq \lambda < 2 \Lambda}| T_{ \mu_{a,q,\tau,\lambda}}|:   \right \rVert_{\ell^p \to \ell^{r^\prime}} \leq A 2^{dk/2} \Lambda^{-2-\eta}  \Lambda^{d(1/r^\prime- 1/p)} \label{Est:0}. 
\end{align}

\end{lemma}
\begin{proof}
We begin by noting the pointwise bounds
\begin{flalign}
 & \sup_{\Lambda \leq \lambda < 2 \Lambda} | f*   \check{\mu}_{a/q, \tau, \lambda} | \leq  |  f * \check{\mu}_{a/q, \tau, \Lambda} |  + \left( \int_{\Lambda}^{2 \Lambda} \frac{d}{d\lambda} | f * \check{\mu}_{a/q, \tau, \lambda}| ^2 d \lambda \right)^{1/2}  \label{Est:Var1.0}& 
 \end{flalign}
 and
 \begin{align}
 &  \left( \int_{\Lambda}^{2 \Lambda} \frac{d}{d\lambda} | f * \check{\mu}_{a,q, \tau, \lambda}| ^2 d \lambda \right)^{1/2}   \label{Est:Var2.0} \\  \leq &  \left( \int_{\Lambda}^{2 \Lambda}  \left | \frac{d}{d\lambda}  f * \check{\mu}_{a,q, \tau, \lambda} \right| ^2 d \lambda \right)^{1/2} \cdot \left( \int_{\Lambda}^{2 \Lambda}  |   f * \check{\mu}_{a,q, \tau, \lambda}| ^2 d \lambda \right)^{1/2}. \nonumber  
\end{align}
Similar to \cite{Magyar}, we may use the definition of $\check{\mu}_{a,q,\tau, \lambda}$ and 
\begin{align*}
e^{-x^2} \leq A (1+x^2)^{-d/2} \qquad \forall x \in \mathbb{R}
\end{align*}
to verify that
 \begin{align}\label{Est:Multsize}
|| \mu_{a,q,\tau,\lambda}|| _{L^\infty(\mathbb{T}^d)} + || \lambda \frac{d}{d \lambda} \mu_{a,q,\tau,\lambda}|| _{L^\infty(\mathbb{T}^d)} \leq A\left[ \frac{\epsilon^2 + \tau^2}{\epsilon} \right]^{d/4}.
 \end{align}
 Combining \eqref{Est:Var1.0}, \eqref{Est:Var2.0}, and \eqref{Est:Multsize} yields
\begin{align}\label{Est:l2}
\left \lVert \sup_{\Lambda \leq \lambda < 2 \Lambda} |T_{\mu_{a,q,\tau,\lambda}}|  \right \rVert_{\ell^2 \to \ell^2} \leq A \left[ \frac{\epsilon^2 + \tau^2}{\epsilon} \right]^{d/4}.
\end{align}
Consequently, estimate \eqref{Est:0} holds at $(p,r)= (2,2)$.
From the kernel bound
\begin{align}\label{Est:Ker3}
\left| \check{\mu}_{a,q,\tau,\lambda} (x) \right| \leq \frac{ A_M }{\Lambda^d}  \left[ \frac{\epsilon^2+\tau^2}{\epsilon^2} \right]^{d/4} \left[ \Lambda^d \phi^1_{\Lambda}(x) \right]^M \qquad \forall \in \mathbb{N}, x \in \mathbb{Z}^d,
\end{align}
where $\phi^1_{\Lambda}$ is again given by \eqref{Def:Bump}, and  Lemma \ref{L:Bump}, it follows that for all $\left( \frac{1}{p} , \frac{1}{r}\right) \in [0,1]^2$ such that $\max\{\frac{1}{p},\frac{1}{r}\} = 1$ 
\begin{align}\label{Est:-1}
\left| \left| \sup_{\Lambda \leq \lambda < 2 \Lambda} |T_{\mu_{a,q,\tau,\lambda}}| \right| \right|_{\ell^p \to \ell^{r^\prime}} \leq A  \left[ \frac{\epsilon^2+\tau^2}{\epsilon^2} \right]^{d/4} \Lambda^{d(1/r^\prime - 1/p)}. 
\end{align}
 Interpolating \eqref{Est:l2} and \eqref{Est:-1} yields for all $(\frac{1}{p}, \frac{1}{r}) \in \mathcal{S}(d)$
\begin{align}\label{Est:-.5}
||\sup_{\Lambda < \lambda < 2 \Lambda} |T_{ \mu_{a,q,\tau,\lambda}}|  || _{\ell^p \to \ell^{r^\prime}} \leq A  \left[ \frac{\epsilon^2 + \tau^2}{\epsilon} \right] \left[ \frac{\epsilon^2+\tau^2}{\epsilon^2} \right]^{d/4-1} \Lambda^{d(1/r^\prime- 1/p)}. 
\end{align}
Substituting $\epsilon = \frac{1}{\lambda^2}$ and $|\tau| \simeq \frac{2^k}{\Lambda^2}$ into \eqref{Est:-.5}  yields \eqref{Est:0}.
\end{proof} 
The next result is used to obtain improving properties for $\sup_{\Lambda \leq \lambda < 2 \Lambda} | \mathscr{B}_\lambda - \mathscr{C}_\lambda |$. 
\begin{lemma}\label{L:Est-gamma-Imp}
For $d \geq 5, q \in \mathbb{N}$, $a \in \mathbb{Z}^\times_q,$ and $\tau \in \mathbb{R}$, let 
\begin{align*}
\gamma_{a/q,\tau,\lambda}(\xi)= \sum_{\ell \in \mathbb{Z}^d/q\mathbb{Z}^d}G(a/q, \ell) \Phi_q (\xi - \ell /q) e^{- \pi |\xi - \ell/q|^2/2 ( \epsilon - i \tau)}. 
\end{align*}
Then for all $(\frac{1}{p}, \frac{1}{r}) \in \mathcal{S}(d)$, there exists $A=A(d,p,r)$ and $\eta=\eta(d,p,r)>0$ such that for all $k \geq 0 , \Lambda \in 2^{\mathbb{N}}, \tau \in I_k (\Lambda)$,  and $f \in \ell^p(\mathbb{Z}^d)$
\begin{align}
\left \lVert \sup_{\Lambda \leq \lambda < 2 \Lambda}  | T_{ \gamma_{a,q,\tau,\lambda} }| \right \rVert_{\ell^p \to \ell^{r^\prime}} \leq A q^{-2-\eta}  2^{dk(1/2-2/d)} \Lambda^{d(1/r^\prime-1/p)}\label{Est:0.}.
\end{align}
\end{lemma}
\begin{proof}
We begin by majorizing
\begin{align}
& \sup_{\Lambda \leq \lambda < 2 \Lambda} | f*   \check{\gamma}_{a,q, \tau, \lambda} |  \nonumber\\ \leq &  |  f * \check{\gamma}_{a,q, \tau, \Lambda} |  + \left( \int_{\Lambda}^{2 \Lambda}  \left | \frac{d}{d\lambda}  f * \check{\gamma}_{a,q, \tau, \lambda} \right| ^2 d \lambda \right)^{1/2} \cdot \left( \int_{\Lambda}^{2 \Lambda}  |   f * \check{\gamma}_{a,q, \tau, \lambda}| ^2 d \lambda \right)^{1/2}  . \label{Est:Var2} 
\end{align}
Using the definition or $\check{\gamma}_{a,q,\tau, \lambda}$, it is straightforward to check that 
\begin{align}
\left \lVert  \gamma_{a,q,\tau, \lambda} \right \rVert_{L^\infty(\mathbb{T}^d)} +    \left \lVert  \lambda \frac{d}{d\lambda}  \gamma_{a,q, \tau, \lambda} \right \rVert_{L^\infty(\mathbb{T}^d)} \leq A q^{-d/2}, \label{Est:Mult2}
\end{align}
so that
\begin{align}\label{Est:l2.}
\left| \left| \sup_{\Lambda \leq \lambda < 2 \Lambda}|T_{\gamma_{a,q,\tau,\lambda}}|  \right| \right| _{\ell^2 \to \ell^2} \leq A q^{-d/2}.
\end{align}
The kernel bound 
\begin{align}\label{Est:Ker4}
\left| \check{\gamma}_{a,q,\tau,\lambda}(x) \right| \leq \frac{A_M}{\Lambda^d} \left[ \frac{\epsilon^2+\tau^2}{\epsilon^2} \right]^{d/4} \left[  \Lambda^d \phi^1_{\Lambda}(x) \right]^M \qquad \forall M \in \mathbb{N}, x \in \mathbb{Z}^d
\end{align}
and Lemma \ref{L:Bump} imply that for all $\left( \frac{1}{p}, \frac{1}{r} \right) \in [0,1]^2$ such that $\max\left( \frac{1}p, \frac{1}r \right) = 1$
\begin{align}\label{Est:l1}
\left \lVert \sup_{\Lambda \leq \lambda < 2 \Lambda} |T_{ \gamma_{a,q,\tau,\lambda}}|  \right \rVert _{\ell^p \to \ell^{r^\prime}} \leq A \left[ \frac{\epsilon^2+\tau^2}{\epsilon^2} \right]^{d/4}  \Lambda^{d(1/r^\prime - 1/p)} . 
\end{align}
Interpolating estimates \eqref{Est:l2.} and \eqref{Est:l1} gives that for all $(\frac{1}{p}, \frac{1}{r}) \in \mathcal{S}(d)$
\begin{align}\label{Est:-2}
\left \lVert \sup_{\Lambda \leq \lambda <2 \Lambda} T_{ \gamma_{a,q,\tau,\lambda}}  \right \rVert_{\ell^p \to \ell^{r^\prime}} \leq A q^{-2-\eta} \left[ \frac{\epsilon^2+\tau^2}{\epsilon^2} \right]^{d/4-1}   \Lambda^{d(1/r^\prime - 1/p)}. 
\end{align}
Substituting $\epsilon = \frac{1}{\Lambda^2}$ and $|\tau| \simeq \frac{2^k}{\Lambda^2}$ into \eqref{Est:-2} yields \eqref{Est:0.}.
\end{proof}
We now prove estimate \eqref{Est:Improving2} in the following result:
\begin{prop}\label{P:Imp-E}
Let $d \geq 5$ and $(\frac{1}{p}, \frac{1}{r}) \in \mathcal{S}(d)$. Then there is $A=A(d,p,r)$ such that for all $\Lambda \in 2^{\mathbb{N}}$
\begin{align}\label{Est:E0}
\left| \left| \sup_{\Lambda \leq \lambda < 2 \Lambda} |\mathscr{R}_\lambda |  \right| \right|_{\ell^p \to \ell^{r^\prime}}  \leq& A \Lambda^{d(1/r^\prime - 1/p)}. 
\end{align}

\end{prop}

\begin{proof}
To verify \eqref{Est:E0}, it is enough to show
\begin{align}
 \left| \left| \sup_{\Lambda \leq \lambda < 2 \Lambda} |\mathscr{A}_\lambda - \mathscr{B}_\lambda |  \right| \right| _{\ell^p \to \ell^{r^\prime}}\leq& A \Lambda^{d(1/r^\prime - 1/p)} \label{Est:AB}\\
\left| \left| \sup_{\Lambda \leq \lambda < 2 \Lambda} |\mathscr{B}_\lambda - \mathscr{C}_\lambda |  \right| \right|_{\ell^p \to \ell^{r^\prime}} \leq& A \Lambda^{d(1/r^\prime - 1/p)} \label{Est:BC}.
\end{align}
 We begin the proof of \eqref{Est:AB} by observing from the definitions \eqref{Def:a(lambda)} and \eqref{Def:b(lambda)} that there is a constant $\mathfrak{C}>0$ such that for each
$a \in \mathbb{Z}^\times_q, 1 \leq  q \leq \Lambda$, and $\Lambda \in 2^{\mathbb{N}}$
\begin{align}
& \left \lVert \sup_{\Lambda \leq \lambda < 2 \Lambda} | \mathscr{A}^{a/q} _\lambda - \mathscr{B}^{a/q}_\lambda  |  \right \rVert _{\ell^p \to \ell^{r^\prime}} \nonumber \\ \leq &  A \Lambda^{2-d} \sum_{k=0}^{\log_2(\Lambda/Q) + \mathfrak{C}} \int_{I_k(\Lambda)}  \frac{ \left \lVert \sup_{\Lambda \leq \lambda < 2 \Lambda} | T_{ \mu_{a,q, \tau, \lambda}} |  \right \rVert_{\ell^p \to \ell^{r^\prime}} }{(\epsilon^2 + \tau)^{d/4}} d \tau.  \label{Est:TrianRl1}
\end{align}
By Lemma \ref{L:Est-mu-Imp}, the last line of the above display is majorized by 
\begin{align*}
&  A \Lambda^{-d+2} \sum_{k=0}^{\log_2(\Lambda/q)+\mathfrak{C}} \left[ \frac{2^k}{\Lambda^2} \right]^{1-d/2} 2^{dk/2} \Lambda^{-2} \left[\Lambda^{d(1/r^\prime- 1/p)} || f||_{\ell^p(\mathbb{Z}^d)} \right]
\\\leq& A \Lambda^{-d+2} \Lambda^{d-4} \frac{\Lambda}{q}  \left[\Lambda^{d(1/r^\prime- 1/p)} ||f||_{\ell^p(\mathbb{Z}^d)} \right] \\ \leq &  A \frac{1}{q \Lambda}  \left[ \Lambda^{d(1/r^\prime- 1/p)} ||f||_{\ell^p(\mathbb{Z}^d)} \right] . 
\end{align*}
Summing on $a \in \mathbb{Z}^\times_q $ and $ 1 \leq q \leq \Lambda$ then yields 
\begin{align*}
& \left| \left| \sup_{\Lambda \leq \lambda < 2 \Lambda} |(\mathscr{A}_\lambda - \mathscr{B}_\lambda) f| \right| \right|_{\ell^{r^\prime}(\mathbb{Z}^d)}\\  \leq& \sum_{q=1}^\Lambda \sum_{a \in \mathbb{Z}^\times_q} \left| \left| \sup_{\Lambda \leq \lambda < 2 \Lambda} |(\mathscr{A}^{a/q}_\lambda - \mathscr{B}^{a/q}_\lambda) f| \right| \right|_{\ell^{r^\prime}(\mathbb{Z}^d)} \\ \leq& A  \Lambda^{d(1/r^\prime-1/p)} || f||_{\ell^p(\mathbb{Z}^d)}. 
\end{align*}
It remains to handle the contribution from $\mathscr{B}_\lambda - \mathscr{C}_\lambda$. To this end, \eqref{Def:b(lambda)} and \eqref{Def:c(lambda)} ensure that there is a constant $\mathfrak{C}>0$ such that for each $a \in \mathbb{Z}^\times_q$, $1 \leq  q \leq \Lambda$, and $\Lambda \in 2^{\mathbb{N}}$
\begin{align*}
 & \left| \left|  \sup_{\Lambda \leq \lambda < 2 \Lambda} \left| ( \mathscr{B}_\lambda^{a/q} - \mathscr{C}_\lambda^{a/q}) f  \right| \right| \right|_{\ell^{r^{\prime}}(\mathbb{Z}^d)} \\ \leq&  A \Lambda^{-d+2} \sum_{k= \log_2(\Lambda/q)-\mathfrak{C}}^{\infty}  \int_{I_k(\Lambda)} \frac{ || T_{\gamma_{a,q,\tau,\lambda}} f  ||_{\ell^{r^\prime}(\mathbb{Z}^d)}}{(\epsilon^2 + \tau^2) ^{d/4} }d \tau 
 \end{align*}
 By Lemma \ref{L:Est-gamma-Imp}, the last line of the above display can be bounded by
 \begin{align*}
 & A\Lambda^{-d+2} \sum_{k= \log_2(\Lambda/q)-\mathfrak{C}}^{\infty}  \left[ \frac{2^k}{\Lambda^2}\right]^{1-d/2}  q^{-2}  2^{dk(1/2-2/d)} \left[ \Lambda^{d(1/r^\prime-1/p)}  || f||_{\ell^p(\mathbb{Z}^d)} \right]\\ \leq &  A \sum_{k= \log_2(\Lambda/q)-\mathfrak{C}}^{\infty}  2^{-k}   q^{-2}  \left[ \Lambda^{d(1/r^\prime-1/p)}  || f||_{\ell^p(\mathbb{Z}^d)} \right] \\ \leq&  \frac{A}{q \Lambda}   \left[ \Lambda^{d(1/r^\prime-1/p)}  || f||_{\ell^p(\mathbb{Z}^d)} \right]. 
\end{align*}
Summing on $a \in \mathbb{Z}^\times_q$ and $1 \leq q \leq \Lambda$ yields an upper bound $O\left(  \left[ \Lambda^{d(1/r^\prime-1/p)} || f||_{\ell^p(\mathbb{Z}^d)}\right]\right)$.

\end{proof}
  \begin{prop}\label{Cor:Imp-Diag}
 
For all $d \geq 5$ and $(\frac{1}{p}, \frac{1}{r}) \in \mathcal{R}(d)$, there is $A=A(d,p,r)$ such that for all $\Lambda \in 2^{\mathbb{N}}$
\begin{align*}
\left| \left| \sup_{\Lambda \leq \lambda < 2 \Lambda} |\mathscr{A}_\lambda | \right| \right|_{\ell^p \to \ell^{r^\prime}} \leq& A \Lambda^{d(1/r^\prime - 1/p)}. 
\end{align*}

  \end{prop}
  \begin{proof}
  By Propositions \ref{P:Imp-M} and \ref{P:Imp-E}, it follows that for all $d \geq 5, (\frac{1}{p}, \frac{1}{r}) \in \mathcal{
  S}(d)$ there is $A=A(d,p,r)$ such that for all $\Lambda \in 2^{\mathbb{N}}$ 
\begin{align}\label{Est:Diag}
\left| \left| \sup_{\Lambda \leq \lambda < 2 \Lambda} |\mathscr{A}_\lambda | \right| \right|_{\ell^p \to \ell^{r^\prime}} \leq& A \Lambda^{d(1/r^\prime - 1/p)}. 
\end{align}
Interpolating estimate \eqref{Est:Diag} with the trivial $\ell^\infty \to \ell^\infty$ bound for $\sup_{\Lambda \leq \lambda < 2 \Lambda} |\mathscr{A}_\lambda|$ yields the Proposition. 
  \end{proof}

\section{Sparse Domination for $\sup_\lambda |\mathscr{C}_\lambda|$}
Our goal in this section is to prove estimate \eqref{Est:MS}, where $\mathscr{C}_\lambda : f \mapsto f * \check{c}_\lambda$ and $c_\lambda$ is given in \eqref{Def:c(lambda)}. To this end, we need to state a restricted weak-type sparse result, which first appears in \cite{Kesler10}. We include an original, self-contained proof for convenience.

\begin{theorem}\label{Restr-Sparse}
Let $T$ be an operator on $\mathbb{Z}^d$ satisfying the property that for some $p,r: \frac{1}{p} + \frac{1}{r} >1$  there is an $A$ such that for all finite sets $E_1, E_2 \subset \mathbb{Z}^d$ and $|f| \leq 1_{E_1}, |g| \leq 1_{E_2}$, there is a sparse collection $\mathcal{S}$ such that 
\begin{align*}
    \left| \langle T f, g\rangle \right| \leq A \Lambda_{\mathcal{S}, p,r} (1_{E_1}, 1_{E_2}).
\end{align*} Then for every $\tilde{p}> p, \tilde{r}>r$ such that $\frac{1}{\tilde{p}} + \frac{1}{\tilde{r}} >1$ there is $A$ such that for all finitely supported $f,g : \mathbb{Z}^d \to \mathbb{C}$ there is a sparse collection $\mathcal{S}$ such that
\begin{align*}
    \left| \langle T f , g \rangle \right| \leq A \Lambda_{\mathcal{S}, \tilde{p},\tilde{r}} (f, g). 
\end{align*} 
\end{theorem}
The assumption of Theorem \ref{Restr-Sparse} is referred to as a restricted weak-type sparse bound on $T$. The conclusion allows us to upgrade the restricted weak-type bound to a standard sparse bound, at the cost of raising the averaging exponents $p,r$ by an arbitrarily small amount. 
\begin{proof}
Fix $f,g : \mathbb{Z}^d \to \mathbb{C}$ supported on a cube $3E$ where $E$ is dyadic. Without loss of generality, suppose $|f| , |g| \leq 1$ and decompose 
\begin{align*}
 f = \sum_{k \geq 0} 2^{-k} f_k ,~~~ g = \sum_{l \geq 0} 2^{-l} g_l 
 \end{align*}
where $f_k = 2^k f 1_{\{2^{-{k+1}} < |f|  \leq 2^{-k} \}},~~~ g_l = 2^l g 1_{\{2^{-{l+1}} < |g|  \leq 2^{-l} \}} $. Then by assumption 
\begin{align*}
\left| \langle Tf ,g \rangle \right| \leq& \sum_{k, l \geq 0} 2^{-k-l} \left \langle \left| T 1_{\{2^{-{k+1}} < f \leq 2^{-k}\}} \right| , 1_{\{2^{-{l+1}} < g \leq 2^{-l}\}} \right \rangle \\ \leq& A \sum_{k, l \geq 0} 2^{-k-l} \Lambda_{\mathcal{S}_{k,l}, p,r}(1_{\{2^{-{k+1}} < f \leq 2^{-k}\}}, 1_{\{2^{-{l+1}} < g \leq 2^{-l}\}} ). 
\end{align*}
For $\mu_1, \mu_2 \geq 0$, let $\mathcal{Q}_{\mu_1, \mu_2}(k,l) := \mathcal{Q}^1_{\mu_1}(k,l) \cap \mathcal{Q}^2_{\mu_2}(k,l)$, where 
\begin{align*}
\mathcal{Q}^1_{\mu_1}(k,l):=& \left\{ Q \in \mathcal{S}_{k,l}: 2^{-{\mu_1 - 1}} < \frac{ |Q \cap \{ 2^{-{k+1}} < |f| \leq 2^{-k}\} |}{|Q|} \leq 2^{-\mu_1} \right\} \\
\mathcal{Q}^2_{\mu_2}(k,l) :=&  \left\{ Q \in \mathcal{S}_{k,l}:  2^{-{\mu_2 - 1}} < \frac{ |Q \cap \{ 2^{-{l+1}} < |f| \leq 2^{-l}\} |}{|Q|} \leq 2^{-\mu_2} \right\}.
\end{align*}
It suffices to produce a sparse collection $\mathcal{S}(f,g)$ such that for every $\mu_1 , \mu_2 \geq 0$ and $\tilde{p} > p, \tilde{r} >r$
\begin{align*}
 \sum_{Q \in \mathcal{Q}_{\mu_1, \mu_2}} \langle f \rangle_{Q,\bar{p}} \langle g \rangle_{Q,\bar{r}} |Q|  \leq  A \sum_{Q \in \mathcal{S}(f,g)} \langle f \rangle_{Q,\bar{p}} \langle g \rangle_{Q, \bar{r}} |Q|. 
\end{align*}
The first generation is denote by $\mathcal{S}_1(f,g)$ and is set equal to the maximal shifted dyadic cubes $Q \subset 3E$ such that 
\begin{align*} \langle f \rangle_{Q, \bar{p}} \geq A_0 \langle f \rangle_{3E,\bar{p}}~or~ \langle g \rangle_{Q, \bar{r}} \geq A_0 \langle g \rangle_{3E,\bar{r}}.
\end{align*}
For large enough constant $A_0$, $\left| \bigcup_{\mathcal{S}(f,g)} Q \right| \leq \frac{|E|}{100}$. For each $Q \in \mathcal{S}_1(f,g)$, we choose $R \in \mathcal{S}_2(f,g)$ provided it is a maximal shifted dyadic cube inside $Q$ such that
\begin{align*}
 \langle f \rangle_{R, \bar{p}} \geq A_0 \langle f \rangle_{Q,\bar{p}}~or~ \langle g \rangle_{R, \bar{r}} \geq A_0 \langle g \rangle_{Q,\bar{r}}.
\end{align*} 
For large enough constant $A_0$, $\left| \bigcup_{R \in \mathcal{S}_2(f,g) \subset Q} R \right| \leq \frac{|Q|}{100}$ for all $Q \in \mathcal{S}_1(f,g)$. Iterating this procedure a finite number of times yields the desired sparse collection of cubes $\mathcal{S}(f,g) = \{3E\} \cup \bigcup_{m=1}^{k_0(E)} \mathcal{S}_m(f,g)$.   
Next, we may suppose without loss of generality that the cubes $\mathcal{Q}_{\mu_1, \mu_2}$ are dyadic and set for each $m \geq 1$
\begin{align*}
 \mathcal{Q}_{\mu_1, \mu_2, m}= \{Q \in \mathcal{Q}_{\mu_1, \mu_2} : \min \{ l : \exists R \in \mathcal{S}_l(f,g): R \supset Q\} = m\}.
 \end{align*}
If there is no $R \in \mathcal{S}(f,g)$ for which $R \supset Q$, then assign $Q \in \mathcal{Q}_{\mu_1, \mu_2,0}$. By construction, 
\begin{flalign*}
& \sum_{Q \in \mathcal{Q}_{\mu_1, \mu_2}} \langle f \rangle_{Q,\bar{p}} \langle g \rangle_{Q,\bar{r}} |Q| & \\ =& \sum_{m=0}^{k_0} \sum_{Q \in \mathcal{Q}_{\mu_1, \mu_2,m}} \langle f \rangle_{Q,\bar{p}} \langle f \rangle_{Q,\bar{r}} |Q| & \\ \leq& A  2^{-\mu_1 /\bar{p}} 2^{-\mu_2 / \bar{r}} \sum_{m=0}^{k_0}\sum_{R \in \mathcal{S}_m(f,g)}   ~~~\sum_{\substack{ k,l \geq 0 \\ \mathcal{Q}^1_{\mu_1}(k,l) \cap \mathcal{Q}^2_{\mu_2}(k,l) \cap \mathcal{Q}_{\mu_1,\mu_2, m} \not = \emptyset}} 2^{-k} 2^{-l} \sum_{\substack{ Q \subset R  \\ Q  \in \mathcal{Q}^1_{\mu_1}(k,l) \cap \mathcal{Q}^2_{\mu_2}(k,l)}} |Q|.&
\end{flalign*}
Note that because $\mathcal{S}_{k,l}$ is a sparse collection for each $k,l \geq 0$, 
\begin{align*} 
\sum_{\substack{ Q \subset R \\ Q  \in \mathcal{Q}^1_{\mu_1}(k,l) \cap \mathcal{Q}^2_{\mu_2}(k,l)}} |Q| \leq A |R|.
\end{align*}
If $\mathcal{Q}^1_{\mu_1}(k,l) \cap \mathcal{Q}^2_{\mu_2}(k,l) \cap \mathcal{Q}_{\mu_1,\mu_2, m} \not = \emptyset$ for some $m \geq 1$, then any cube $Q \in \mathcal{Q}^1_{\mu_1}(k,l) \cap \mathcal{Q}^2_{\mu_2}(k,l) \cap \mathcal{Q}_{\mu_1,\mu_2, m}$ such that $Q \subset R$ for $R \in \mathcal{S}_m (f,g)$ satisfies
\begin{align*}
2^{-k} 2^{-\mu_1/ \bar{p}} \leq& A \langle f_k \rangle_{Q,\bar{p}} \leq A \langle f \rangle_{Q, \bar{p}} \leq A \langle f \rangle_{R, \bar{p}} \\
2^{-l} 2^{-\mu_2/ \bar{r}}  \leq& A \langle g_l \rangle_{Q, \bar{r}} \leq A \langle g \rangle_{Q, \bar{r}} \leq A \langle g \rangle_{R, \bar{r}}.
\end{align*}
If $\mathcal{Q}^1_{\mu_1}(k,l) \cap \mathcal{Q}^2_{\mu_2}(k,l) \cap \mathcal{Q}_{\mu_1,\mu_2, 0} \not = \emptyset$, then any cube $Q \in \mathcal{Q}^1_{\mu_1}(k,l) \cap \mathcal{Q}^2_{\mu_2}(k,l) \cap \mathcal{Q}_{\mu_1,\mu_2, 0}$ satisfies
\begin{align*}
2^{-k} 2^{-\mu_1/ \bar{p}} \leq& A \langle f_k \rangle_{Q,\bar{p}} \leq A \langle f \rangle_{Q, \bar{p}} \leq A \langle f \rangle_{3E, \bar{p}} \\
2^{-l} 2^{-\mu_2/ \bar{r}}  \leq& A \langle g_l \rangle_{Q, \bar{r}} \leq A \langle g \rangle_{Q, \bar{r}} \leq A \langle g \rangle_{3E, \bar{r}}.
\end{align*}
Therefore, 
\begin{align*}
& \sum_{Q \in \mathcal{Q}_{\mu_1, \mu_2}} \langle f \rangle_{Q,\bar{p}} \langle g \rangle_{Q,\bar{r}} |Q| \\ \leq & A 2^{-\mu_1 /\bar{p}} 2^{-\mu_2 / \bar{r}} \sum_{m=0}^{k_0}\sum_{R \in \mathcal{S}_m(f,g)} \sum_{\substack{ k,l \geq 0 \\ 2^{-k} \leq A 2^{\mu_1/ \bar{p}} \\ 2^{-l} \leq A 2^{\mu_2 / \bar{r}}}}  2^{-k} 2^{-l} \left[ \sum_{\substack{ Q \subset R \\ Q  \in \mathcal{Q}_{\mu_1, \mu_2,m}}} |Q| \right] \\ \leq & A \sum_{R \in \mathcal{S}(f,g)} \langle f \rangle_{R, \bar{p}} \langle g \rangle_{R, \bar{r}} |R|. 
\end{align*}

\end{proof}
We now restate estimate \eqref{Est:MS} as a stand-alone result and then prove it.

\begin{theorem}\label{Thm:Sp-M}
Let $d \geq 5$ and $\left( \frac{1}{p}, \frac{1}{r} \right) \in \mathcal{S}(d)$. Then 
\begin{align*}
\left|\left| \sup_{\lambda \in \tilde{\Lambda}} | \mathscr{C}_\lambda| : (p,r) \right| \right| <\infty. 
\end{align*}

\end{theorem}
\begin{proof}
It suffices to prove Theorem  under the additional restriction $\frac{d}{d-2}< p \leq 2$. In particular, it is enough to prove the conclusion of Theorem \ref{Thm:Sp-M} for $(\frac{1}p, \frac{1}r)$ near $( \frac{d-2}d, \frac{d-2}d)$ because the result is strongest there. To proceed, we recall that for any $\# \in 2^{\mathbb{Z}}$, the operator $P_{\leq \#}$ is defined by
$$
\widehat{ P_{\leq \#}(f)}(\xi) = \sum_{\ell \in \mathbb{Z}^d / q \mathbb{Z}^d} \sum_{2^k \leq \#} \psi_k(\xi-\ell/q) \widetilde{\Phi}_q(\xi- \ell/q) \hat{f}(\xi) \qquad \forall \xi \in [-1/2, 1/2)^d
$$
where $\widetilde{\Phi}_q$ is given in \eqref{Exp:Decomp}. 
Then we obtain by the triangle inequality 
\begin{align*}
 \sup_{\Lambda \in 2^\mathbb{N}} \sup_{\Lambda \leq \lambda < 2 \Lambda} \left| \mathscr{C}_{\lambda}  f  \right|  \leq&   \sup_{\Lambda \in 2^\mathbb{N}} \sup_{\Lambda \leq \lambda < 2 \Lambda} \left| \mathscr{C}_{\lambda} P_{\leq  1/ \Lambda} f  \right|  \\ +& \sum_{  N \in 2^{\mathbb{N}}  } \sup_{\Lambda \geq Nq } \sup_{\Lambda \leq \lambda < 2 \Lambda} \left| \mathscr{C}_{\lambda}   P_{N/\Lambda} f  \right|.
\end{align*}
We first focus our attention on obtaining $\eta=\eta(d,p,r)>0$ such that for all $q \in \mathbb{N}, a \in \mathbb{Z}^\times_q$
\begin{align}\label{Est:centered}
 \left| \left|  \sup_{\Lambda} \sup_{\Lambda \leq \lambda < 2 \Lambda} \left| \mathscr{C}^{a/q}_{\lambda}   P_{\leq 1/\Lambda}  \right| : (p,r) \right| \right|\leq A q^{-2-\eta}.
\end{align}
By Theorem \ref{Restr-Sparse}, it suffices to obtain for all $(\frac{1}{p}, \frac{1}{r})$ satisfying $\max\{\frac{1}{p}, \frac{1}{q} \} < \frac{d-2}{d}$ and arbitrarily close to $(\frac{d-2}{d}, \frac{d-2}d)$ some $\eta = \eta(d,p,r)>0$ such that
\begin{align}\label{Est:Goal1} \left| \left|  \sup_{\Lambda} \sup_{\Lambda \leq \lambda < 2 \Lambda} \left| \mathscr{C}^{a/q}_{\lambda}   P_{\leq 1/\Lambda}  \right| : (p,r) \right| \right|_{restricted} \leq A q^{-2-\eta}
\end{align}
where the sparse restricted norm $||T:(p,q)||_{restricted}$ is defined to be the infimum over all $C>0$ such that $\forall f,g : \mathbb{Z}^d \to \mathbb{C}~s.t.~|f|\leq 1_{E_1}, |g| \leq 1_{E_2}, \max\{|E_1|, |E_2|\} < \infty$, the estimate $|\langle Tf , g \rangle| \leq C \sup_{\mathcal{S}} \Lambda_{\mathcal{S}, p,r} (1_{E_1}, 1_{E_2})$ holds. To this end, let $f,g:\mathbb{Z}^d \rightarrow \mathbb{C}$ be finitely supported on $3E$ where $E$ is a dyadic cube. Now let $\mathcal{Q}(E)$ be the maximal dyadic cubes satisfying the conditions 
  \begin{align*}
\langle 1_{E_1} \rangle_{3Q,1} \geq& A_0 \langle 1_{E_1} \rangle_{3E,1} \\ 
 \left\langle \sup_{\Lambda} \sup_{\Lambda \leq \lambda < 2 \Lambda} \left| \mathscr{C}^{a/q}_{\lambda}   P_{\leq 1/\Lambda} f\right| \right \rangle_{Q,p} \geq & A_0 q^{-2-\eta} \langle f \rangle_{3E,p}
\end{align*} 
so that $\left| \bigcup_{J \in \mathcal{Q}(E)} J \right| < \frac{|E|}{100}$ for a large enough constant $A_0$. We first majorize
\begin{align*}
& \left\langle \sup_{\Lambda} \sup_{\Lambda \leq \lambda < 2 \Lambda} \left| \mathscr{C}^{a/q}_{\lambda}   P_{\leq 1/\Lambda} f\right|, g \right \rangle\\  \leq & \sum_{Q \in \mathcal{Q}(E)} \left \langle  1_Q \sup_{\Lambda \leq \ell(Q)} \sup_{\Lambda \leq \lambda < 2 \Lambda} \left| \mathscr{C}^{a/q}_{\lambda}   P_{\leq 1/\Lambda} (1_{3Q}f)\right|, g \right \rangle \\ +&  \sum_{Q \in \mathcal{Q}(E)} \left \langle  1_Q \sup_{\Lambda \leq \ell(Q)} \sup_{\Lambda \leq \lambda < 2 \Lambda} \left| \mathscr{C}^{a/q}_{\lambda}   P_{\leq 1/\Lambda} (1_{(3Q)^c}f)\right|, g \right \rangle\\  +& \sum_{Q \in \mathcal{Q}(E)} \left \langle  1_Q \sup_{\Lambda> \ell(Q)} \sup_{\Lambda \leq \lambda < 2 \Lambda} \left|  \mathscr{C}^{a/q}_{\lambda}   P_{\leq 1/\Lambda} f\right|, g \right \rangle \\ +& \left \langle  1_{(\bigcup_{\mathcal{Q}(E)} Q)^c} \sup_{\Lambda} \sup_{\Lambda \leq \lambda < 2 \Lambda} \left|  \mathscr{C}^{a/q}_{\lambda}   P_{\leq 1/\Lambda} f\right|, g\right \rangle  \\ =& \sum_{Q \in \mathcal{Q}(E)} I_Q + \sum_{Q \in \mathcal{Q}(E)} II_Q + \sum_{Q \in \mathcal{Q}(E)} III_Q + IV
\end{align*}
and proceed to obtain satisfactory bounds for each of the above terms separately. First note the pointwise bound
$$1_{(\bigcup_{\mathcal{Q}(E)} Q)^c} \sup_{\Lambda} \sup_{\Lambda \leq \lambda < 2 \Lambda} \left|  \mathscr{C}^{a/q}_{\lambda}   P_{\leq 1/\Lambda} f \right| \leq A q^{-2-\eta} \langle f \rangle_{3E,p}$$
by construction of the stopping time. Therefore, $IV \leq A q^{-2-\eta} \langle f \rangle_{3E,p} \langle g \rangle_{3E,1} |E|$. Next, we may observe from \eqref{Est:Ker11} and the stopping conditions  the pointwise bound
\begin{align}\label{Est:Aux18}
\sum_{Q \in \mathcal{Q}(E)}   1_Q \sup_{\Lambda> \ell(Q)} \sup_{\Lambda \leq \lambda < 2 \Lambda} \left|  \mathscr{C}^{a/q}_{\lambda}   P_{\leq 1/\Lambda} f\right| \leq A q^2 \langle 1_{E_1} \rangle_{3E,1}. 
\end{align}
From estimate \eqref{Est:maxl20}, it follows that
\begin{align}\label{Est:Aux19}
\left\langle \sum_{Q \in \mathcal{Q}(E)}   1_Q \sup_{\Lambda> \ell(Q)} \sup_{\Lambda \leq \lambda < 2 \Lambda} \left|  \mathscr{C}^{a/q}_{\lambda}   P_{\leq 1/\Lambda} f\right| \right \rangle_{3E,2}  \leq A q^{-d/2} \langle f \rangle_{3E,2}.
\end{align}
From \eqref{Est:Aux18} and \eqref{Est:Aux19}, we may observe
\begin{align}\label{Est:Aux20}
\left \langle \sum_{Q \in \mathcal{Q}(E)}   1_Q \sup_{\Lambda> \ell(Q)} \sup_{\Lambda \leq \lambda < 2 \Lambda} \left|  \mathscr{C}^{a/q}_{\lambda}   P_{\leq 1/\Lambda} f\right|\right \rangle_{3E,r^\prime} \leq A q^{-2-\eta} \langle 1_{E_1} \rangle_{3E,p}.
\end{align}
Estimate \eqref{Est:Aux20} combined with H\"{o}lder's inequality implies 
$$\sum_{Q \in \mathcal{Q}(E)} III_Q \leq A q^{-2-\eta} \langle 1_{E_1} \rangle_{3E,p} \langle g \rangle_{3E,r} |E|.$$ 
As we shall be able to recurse on $ \sum_{Q \in \mathcal{Q}(E)} I_Q$ by letting each $Q \in \mathcal{Q}(E)$ play the role that $E$ played in the initial stage, it suffices to obtain
\begin{align}\label{Est:Goal2}
\sum_{Q \in \mathcal{Q}(E)} II_Q \leq A q^{-2-\eta} \langle 1_{E_1} \rangle_{3E,p} \langle g \rangle_{3E,r} |E|.
\end{align}
To this end, we observe from the pointwise bound \eqref{Est:M1} and stopping conditions that 
\begin{align}\label{Est:Aux21}
\sum_{Q \in \mathcal{Q}(E)}  1_Q \sup_{\Lambda> \ell(Q)} \sup_{\Lambda \leq \lambda < 2 \Lambda} \left|  \mathscr{C}^{a/q}_{\lambda}   P_{\leq 1/\Lambda} (1_{(3Q)^c}f) \right| \leq A q^2 \langle 1_{E_1} \rangle_{3E,1}. 
\end{align}
Furthermore, estimate \eqref{Est:maxl20} ensures
\begin{align}\label{Est:Aux22}
& \left\langle \sum_{Q \in \mathcal{Q}(E)}   1_Q \sup_{\Lambda> \ell(Q)} \sup_{\Lambda \leq \lambda < 2 \Lambda} \left|  \mathscr{C}^{a/q}_{\lambda}   P_{\leq 1/\Lambda} (1_{(3Q)^c}f)\right| \right \rangle_{3E,2} \\ \leq & \left\langle \sum_{Q \in \mathcal{Q}(E)}   1_Q \sup_{\Lambda> \ell(Q)} \sup_{\Lambda \leq \lambda < 2 \Lambda} \left|  \mathscr{C}^{a/q}_{\lambda}   P_{\leq 1/\Lambda} (1_{3Q}f)\right| \right \rangle_{3E,2} \nonumber \\ +& \left\langle \sum_{Q \in \mathcal{Q}(E)}   1_Q \sup_{\Lambda> \ell(Q)} \sup_{\Lambda \leq \lambda < 2 \Lambda} \left|  \mathscr{C}^{a/q}_{\lambda}   P_{\leq 1/\Lambda} f\right| \right \rangle_{3E,2} \nonumber \\ \leq & A q^{-d/2} \langle 1_{E_1} \rangle_{3E,2} \nonumber .
\end{align}
From \eqref{Est:Aux21} and \eqref{Est:Aux22}, it follows that 
\begin{align}\label{Est:Aux23}
\left \langle \sum_{Q \in \mathcal{Q}(E)}   1_Q \sup_{\Lambda> \ell(Q)} \sup_{\Lambda \leq \lambda < 2 \Lambda} \left|  \mathscr{C}^{a/q}_{\lambda}   P_{\leq 1/\Lambda} f\right|\right \rangle_{3E,r^\prime} \leq A q^{-2-\eta} \langle 1_{E_1} \rangle_{3E,p}.
\end{align}
Estimate \eqref{Est:Aux23} combined with H\"{o}lder's inequality implies \eqref{Est:Goal2}. Recursing on $\sum_{Q \in \mathcal{Q}(E)} I_Q$ then yields \eqref{Est:Goal1}. 
That 
\begin{align}\label{Est:Goal3}
\left| \left|  \sup_{\Lambda} \sup_{\Lambda \leq \lambda < 2 \Lambda} \left| \mathscr{C}^{a/q}_{\lambda}   P_{N/\Lambda}  \right| : (p,r) \right| \right|_{restricted} \leq A N^{-\eta} q^{-2-\eta}
\end{align}
for all $N \in 2^{\mathbb{N}}, q \in \mathbb{N}$, and $(\frac{1}{p}, \frac{1}{r})$ satisfying $\max\{\frac{1}{p}, \frac{1}{q} \} < \frac{d-2}{d}$ and arbitrarily close to $(\frac{d-2}{d}, \frac{d-2}d)$ and some $\eta = \eta(d,p,r)>0$ follows a very similar argument, and so the details are omitted. Summing \eqref{Est:Goal2} on $a,q$ and \eqref{Est:Goal3} on $a,q$, and $N$ concludes the proof of Theorem \ref{Thm:Sp-M}. 
\end{proof}

\section{Sparse Domination for $\sup_\lambda | \mathscr{R}_\lambda|$}
Our goal is now to obtain estimate \eqref{Est:ES}, which is the sparse bound for $\sup_\lambda |\mathscr{R}_\lambda|$. We proceed by first proving

\begin{lemma}\label{L:Est-T}
For all $d \geq 5, q \in \mathbb{N}$, $a \in \mathbb{Z}^\times_q$, and $(\frac{1}{p}, \frac{1}{r}) \in \mathcal{S}(d)$, there exists $A=A(d,p,r)$ and $\eta=\eta(d,p,r)>0$ such that for all $k \geq 0, \Lambda \in 2^{\mathbb{N}}$, and $\tau \in I_k (\Lambda)$
\begin{align}
\left \lVert \sup_{\Lambda \leq \lambda < 2 \Lambda} |T_{\check{\mu}_{a/q,\tau,\lambda}}| : (p, r) \right \rVert \leq& A 2^{dk/2} \Lambda^{-2-\eta} \label{Est:mu} \\
\left \lVert \sup_{\Lambda \leq \lambda < 2 \Lambda} | T_{\check{\gamma}_{a/q, \tau,\lambda}}| : (p, r) \right \rVert \leq& A q^{-2-\eta}  2^{dk(1/2-2/d)} \label{Est:gamma}.
\end{align}
Here, as elsewhere, $T_m: f \mapsto f * \check{m}$ for a given symbol $m\in L^\infty(\mathbb{T}^d)$. 
\end{lemma}
\begin{proof}
Fix $f,g : \mathbb{Z}^d \rightarrow \mathbb{C}$ finitely supported. Without loss of generality, $g \geq 0$.  Letting $\mathcal{D}_\Lambda$ denote the dyadic cubes with $\ell(Q)=\Lambda$, observe
\begin{align*}
&  \left \langle\sup_{\Lambda \leq \lambda < 2 \Lambda}| f * \check{\mu}_{a/q,\tau,\lambda}|, g \right \rangle  \\ \leq& \sum_{Q \in \mathcal{D}_\Lambda}  \left \langle \sup_{\Lambda \leq \lambda < 2 \Lambda}| f*\check{\mu}_{a/q,\tau,\lambda}|, g 1_Q\right \rangle \\\leq& \sum_{Q \in \mathcal{D}_\Lambda}   \left \langle \sup_{\Lambda \leq \lambda < 2 \Lambda} | (1_Qf) *\check{\mu}_{a/q,\tau,\lambda}|, g 1_Q \right \rangle \\  +  & \sum_{Q \in \mathcal{D}_\Lambda} \sum_{l \geq 1}  \left \langle  \sup_{\Lambda \leq \lambda < 2 \Lambda} | (1_{3^lQ \cap (3^{l-1}Q)^c} f) *\check{\mu}_{a/q,\tau,\lambda}|, g 1_Q \right \rangle  \\ =& \sum_{Q \in \mathcal{D}_\Lambda} A_Q + B_Q. 
\end{align*}
By Lemma \ref{L:Est-mu-Imp}, $A_Q \leq A 2^{dk/2} \Lambda^{-2-\eta(d,p,r)} \langle f \rangle_{Q,p} \langle g \rangle_{Q,r}|Q|$. Moreover, by estimate \eqref{Est:Ker3} and Lemma \ref{L:Bump}, it holds that for all $\left( \frac{1}{p} , \frac{1}{r} \right) \in [0,1]^2$ such that $\max \{ \frac{1}p, \frac{1}r \} =1$ and $M \in \mathbb{N}$
\begin{align}\label{Est:mu-Imp}
& \left \lVert \sup_{\Lambda \leq \lambda < 2 \Lambda} | (1_{3^lQ \cap (3^{l-1}Q)^c} f)*\check{\mu}_{a/q,\tau,\lambda}| \right \rVert_{\ell^{r^\prime}(Q)} \\ \leq & A_M 3^{-Ml} \left[ \frac{\epsilon^2+\tau^2}{\epsilon^2} \right]^{d/4} \Lambda^{d(1/r^\prime - 1/p)} || f||_{\ell^p(3^lQ)} \nonumber. 
\end{align}
Interpolating between estimates \eqref{Est:l2} and \eqref{Est:mu-Imp} ensures that for all $(\frac{1}{p}, \frac{1}{r}) \in \mathcal{S}(d)$ 
\begin{flalign*}
& \left \lVert \sup_{\Lambda \leq \lambda < 2 \Lambda} | (1_{3^lQ \cap (3^{l-1}Q)^c} f) *\check{\mu}_{a/q,\tau,\lambda} | \right \rVert _{\ell^{r^\prime}(Q)} \leq A 3^{-20dl} 2^{dk/2} \Lambda^{-2-\eta} \Lambda^{d(1/r^\prime - 1/p)} || f||_{\ell^p(3^lQ)}.& 
\end{flalign*}
provided we choose $M \geq M_0(d)$. 
From this estimate, it follows that 
\begin{align*}
B_Q \leq A 2^{dk/2} \Lambda^{-2} \sum_{l \in \mathbb{N}} 3^{-10dl} \langle  f \rangle_{3^lQ,p} \langle g \rangle_{Q,r} |Q|.
\end{align*}
Moreover, there is a sparse collection $\mathcal{S}$ for which 
\begin{align*}
\sum_{Q \in \mathcal{D}_\Lambda} \sum_{l \in \mathbb{N}} 3^{-10dl} \langle  f \rangle_{3^lQ,p}  \langle g \rangle_{Q,r}|Q| \leq A \sum_{S \in \mathcal{S}} \langle f \rangle_{S,p} \langle g \rangle_{S,r} |S|. 
\end{align*}
The proof of the estimate involving $\sup_{\Lambda \leq \lambda < 2 \Lambda} | f*\check{\gamma}_{a/q,\tau,\lambda}|$ is very similar, except that Lemma \ref{L:Est-gamma-Imp} and estimate \eqref{Est:l2.} are used in place of Lemma \ref{L:Est-mu-Imp} and estimate \eqref{Est:l2}.

\end{proof}
We now use Lemma \ref{L:Est-T} to deduce 
\begin{lemma}
For all $d \geq 5$ and $(\frac{1}{p}, \frac{1}{r}) \in \mathcal{S}(d)$, there exists $A=A(d,p,r)$ and $\eta=\eta(d,p,r)>0$ such that for all $\Lambda \in 2^{\mathbb{N}}$
\begin{align}
&\left| \left| \sup_{\Lambda \leq \lambda < 2 \Lambda} | \mathscr{A}_\lambda-\mathscr{B}_\lambda | : (p,r) \right| \right| \leq A \Lambda^{-\eta} \label{Est:SpE} \\ 
& \left| \left| \sup_{\Lambda \leq \lambda< 2 \Lambda} | \mathscr{B}_\lambda - \mathscr{C}_\lambda| : (p,r) \right| \right| \leq A \Lambda^{-\eta}\label{Est:SpE2}. 
\end{align}

\end{lemma}

\begin{proof}
Begin by using \eqref{Est:mu} to observe that there is a constant $\mathfrak{C}>0$ and $\eta >0$ such that for every $f,g: \mathbb{Z}^d \rightarrow \mathbb{C}$ finitely supported 
\begin{align*}
&  \left| \left \langle \sup_{\Lambda \leq \lambda < 2 \Lambda} \left| ( \mathscr{A}_\lambda^{a/q} - \mathscr{B}_\lambda^{a/q}) f  \right|, g \right \rangle \right| \\\leq&  A \Lambda^{-d+2} \sum_{k=0}^{ \log_2(\Lambda/q)+\mathfrak{C}}  \int_{I_k(\Lambda)} \frac{ | \langle \sup_{\Lambda \leq \lambda < 2 \Lambda} | f * \check{\mu}_{a/q,\tau,\lambda}|, g \rangle |}{ (\epsilon^2 + \tau^2) ^{d/4}} d \tau \\ \leq&  A \Lambda^{-d+2} \sum_{k=0}^{\log_2(\Lambda/q)+\mathfrak{C}} \left[ \frac{2^k}{\Lambda^2}\right]^{1-d/2} 2^{dk/2} \Lambda^{-2-\eta}   \sup_{\mathcal{S}} \Lambda_{\mathcal{S}, p,r}(f,g) .
\end{align*}
However, $\Lambda^{-d+2} \sum_{k=0}^{\log_2(\Lambda/q)+\mathfrak{C}} \left[ \frac{2^k}{\Lambda^2}\right]^{1-d/2} 2^{dk/2}\Lambda^{-2 - \eta} \leq \frac{A}{q \Lambda^{1+\eta}}$. 
Summing on $a \in \mathbb{Z}^\times_q$ and then $q:1 \leq q \leq \Lambda$ yields an upper bound $O( \Lambda^{-\eta}  \sup_{\mathcal{S}} \Lambda_{\mathcal{S}, p,r}(f,g))$.
To finish, it suffices to note using \eqref{Est:gamma} that is $\mathfrak{C}>0$ and $ \eta >0$ such that for every $f,g : \mathbb{Z}^d \rightarrow \mathbb{C}$ finitely supported
\begin{align*}
 & \left| \left \langle \sup_{\Lambda \leq \lambda < 2 \Lambda} \left| ( \mathscr{B}_\lambda^{a/q} - \mathscr{C}_\lambda^{a/q}) f  \right|, g \right \rangle \right| \\ \leq&  A \Lambda^{-d+2} \sum_{k= \log_2(\Lambda/q)-\mathfrak{C}}^{\infty}  \int_{I_k(\Lambda)}\frac{  | \langle \sup_{\Lambda < \lambda < 2 \Lambda}| f * \check{\gamma}_{a/q,\tau,\lambda}|, g \rangle |}{ (\epsilon^2 + \tau^2) ^{d/4}} d \tau \\ \leq& A \Lambda^{-d+2} \sum_{k= \log_2(\Lambda/q)-\mathfrak{C}}^{\infty} \left[ \frac{2^k}{\Lambda^2} \right]^{1-d/2} q^{-2-\eta} 2^{dk(1/2-2/d)} \sup_{\mathcal{S}} \Lambda_{\mathcal{S}, p,r}(f,g).  
 \end{align*}
 However, $ \Lambda^{-d+2} \sum_{k= \log_2(\Lambda/q)-\mathfrak{C}}^{\infty} \left[ \frac{2^k}{\Lambda^2} \right]^{1-d/2} q^{-2-\eta} 2^{dk(1/2-2/d)}  \leq  \frac{A}{\Lambda q^{1+\eta}}.$  Summing on $a$ and $q$ yields an upper bound $O( \Lambda^{-\eta}  \sup_{\mathcal{S}} \Lambda_{\mathcal{S}, p,r}(f,g))$.

\end{proof}
Summing \eqref{Est:SpE} and \eqref{Est:SpE2} over $\Lambda \in 2^{\mathbb{N}}$ gives 
\begin{prop}\label{P:Sp-E}
For all $d \geq 5$ and $(\frac{1}{p}, \frac{1}{r}) \in \mathcal{S}(d)$
\begin{align*}
\left| \left| \sup_{\lambda \in \tilde{\Lambda}} |\mathscr{R}_\lambda| : (p,r) \right| \right| <\infty. 
\end{align*}
\end{prop}

\begin{theorem}\label{Cor:Diag}
For all $d \geq 5$ and $(\frac{1}{p}, \frac{1}{r}) \in \mathcal{S}(d)$
\begin{align*}
\left| \left| \sup_{\lambda \in \tilde{\Lambda}} |\mathscr{A}_\lambda| : (p,r) \right| \right| <\infty. 
\end{align*}
\end{theorem}
\begin{proof}
Combine Theorem \ref{Thm:Sp-M} and Proposition \ref{P:Sp-E}. 
\end{proof}

\section{Sparse Domination for $\sup_\lambda |\mathscr{A}_\lambda|$}

In addition to Theorem \ref{Restr-Sparse}, we shall need a localized variant of Theorem \ref{Cor:Diag}:
\begin{theorem}\label{Cor:Loc}
Let $d \geq 5$ and $(\frac{1}{r}, \frac{1}{s}) \in \mathcal{S}(d)$. For any collection of cubes $\mathcal{C}$ and $f,g : \mathbb{Z}^d \to \mathbb{C}$ finitely supported, there is a sparse collection of cubes $\mathcal{S}$ such that 
\begin{align*}
\left \langle \sup_{S \supset \mathcal{C}} \left| \mathscr{A}_S f \right|, |g|  \right \rangle \leq A \Lambda_{\mathcal{S}, r,s} (f,g),
\end{align*}
where the supremum is restricted to those spheres $S = \{ y \in \mathbb{Z}^d : |x-y| = \lambda\}$ for which the corresponding ball $B_S = \{ y \in \mathbb{Z}^d : |x-y| \leq \lambda\}$ satisfies $B_S\supset Q$ for some cube $Q\in \mathcal{C}$, and the sparse collection $\mathcal{S}$ satisfies the property that for every cube $Q \in \mathcal{S}$ there is a cube $Q_* \in \mathcal{C}$ such that $Q \supset Q_*$. 

\end{theorem}
\begin{proof}
Retrace the arguments used to show Theorem \ref{Cor:Diag}. 
\end{proof}
The rest of this section is dedicated to showing estimate \eqref{Est:Sp}, which we rewrite as
\begin{theorem}\label{Thm:Sp-Corn}
Let $d \geq 5$ and $\left( \frac{1}{p}, \frac{1}{r} \right) \in \mathcal{R}(d)$. Then 
\begin{align*}
 \left| \left| \sup_{\lambda \in \tilde{\Lambda}} |\mathscr{A}_\lambda| : (p,r)\right| \right| < \infty. 
\end{align*}
\end{theorem}
There are two difficulties in the sparse setting that complicate the proof of Theorem \ref{Thm:Sp-Corn}. The first is that there is no general sparse interpolation machinery. The second is that there is no sparse bound at $(0,1)$, as this point does not break the duality condition. Any successful argument that extends sparse bounds from $\mathcal{S}(d)$ to $\mathcal{R}(d)$ must work with localized sparse bounds for $\sup_{\lambda \in \tilde{\Lambda}} | \mathscr{A}_\lambda|$ near $(\frac{d-2}{d}, \frac{d-2}d)$ and appropriately leverage the trivial $\ell^\infty \to \ell^\infty$ estimate.

\begin{proof}{[Theorem \ref{Thm:Sp-Corn}]}
By Theorem \ref{Restr-Sparse}, it suffices to prove a restricted weak-type sparse bound in a small neighborhood of the line connecting $(0,1)$ with $(\frac{d-2}{d}, \frac{d-2}{d})$ intersected with $\mathcal{R}(d)$. In particular, we shall fix $(\frac{1}{p_1}, \frac{1}{p_2})$ close to $(\frac{d-2}{d}, \frac{d-2}{d})$ and prove sparse esitmates along the line connecting $(\frac{d-2}{d}, \frac{d-2}{d})$ to $(0,1)$. So fix $\left( \frac{1}{p}, \frac{1}{r}\right) \in \mathcal{R}(d)$ on this line and $f,g : \mathbb{Z}^d \to \mathbb{C}, |f| \leq 1_{E_1},|g| \leq 1_{E_2}$ and supported in $3E$ for some dyadic cube $E$. Let the first sparse generation of cubes $\mathcal{Q}(E)$ be those maximal dyadic cubes with respect to the properties 
\begin{align*}
 \langle f \rangle_{3Q, p_1} \geq& A_0 \langle f \rangle_{3E,p_1} \\ 
 \langle g \rangle_{3Q, r_1} \geq& A_0 \langle g \rangle_{3E,r_1} \\ 
 \left \langle \sup_{\lambda \in \tilde{\Lambda}} | \mathscr{A}_\lambda f | \right \rangle_{3Q, p_1} \geq& A_0 \langle f \rangle_{3E,p_1}.
 \end{align*}
For large enough constant $A_0, \left| \bigcup_{Q \in \mathcal{Q}(E) } Q \right| \leq \frac{|E|}{100}$. 
The restricted weak-type sparse bound is therefore reduced to dominating 
\begin{align*}
\left \langle \sum_{Q \in \mathcal{Q}(E)} 1_Q \sup_{\lambda \in \tilde{\Lambda}} | \mathscr{A}_\lambda f | , g \right \rangle =& \left \langle \sum_{Q \in \mathcal{Q}(E)}  1_Q \sup_{S \not \subset 3Q} | \mathscr{A}_S f | , g \right \rangle \\ +& \left \langle \sum_{Q \in \mathcal{Q}(E)}  1_Q \sup_{S \subset 3Q} | \mathscr{A}_S f | , g  \right \rangle \\ =& I + II. 
\end{align*}
As we may recurse on the term $II$, it suffices to bound term $I$ by $A \langle 1_{E_1} \rangle_{3E,p} \langle 1_{E_2} \rangle_{3E,r} |E|$. To this end, estimate using Corollary \ref{Cor:Loc} with $\mathcal{B} = \mathcal{Q}(E)$ that
\begin{align*}
\left \langle \sum_{Q \in \mathcal{Q}(E)}  1_Q \sup_{ S \not \subset 3Q } | \mathscr{A}_S f | , |g| \right \rangle \leq &  \left \langle  \sup_{S \supset \mathcal{Q}(E)} | \mathscr{A}_S f | , |g| \right \rangle \\ \leq & A \sum_{Q \in \mathcal{S}} \langle f \rangle_{Q, p_1} \langle g \rangle_{Q, r_1} |Q|,
\end{align*}
where the supremum is restricted to those discrete spheres $S = \{ y : = |x-y| = \lambda\}$ for which the corresponding ball $B_S = \{ y : |x-y| \leq \lambda\}$ satisfies $B_S\supset R$ for some ball $R \in \mathcal{Q}(E)$, and the sparse collection $\mathcal{S}$ satisfies the property that for all $Q \in \mathcal{S}$ there is $R \in \mathcal{Q}(E)$ such that $Q \supset R.$ So, for each $Q \in \mathcal{S}$, $\langle f \rangle_{Q, p_1} \leq \langle f \rangle_{3E, p_1}, \langle g \rangle_{Q, r_1} \leq \langle g \rangle_{3E, r_1}$ and 
\begin{align*}
 \sum_{Q \in \mathcal{S}} \langle f \rangle_{Q, p_1} \langle g \rangle_{Q, p_2} |Q| \leq& A \langle f \rangle_{3E, p_1} \langle g \rangle_{3E, p_2}  \sum_{Q \in \mathcal{S}: |Q| \leq |E|, Q \cap 3E \not = \emptyset}  |Q|  \\+&  \sum_{Q \in \mathcal{S} : |Q|> |E|, Q \cap 3E \not = \emptyset} || f||_{\ell^{p_1}(3E)} || g  ||_{\ell^{r_1}(3E)} |Q|^{1-\frac{1}{p_1} - \frac{1}{r_1}} \\ &\leq A  \langle f \rangle_{3E,p_1} \langle g \rangle_{3E, r_1} |E|. 
\end{align*}
However, we have the following trivial estimate

\begin{align*} \left \langle \sum_{Q \in \mathcal{Q}(E)}  1_Q \sup_{S \not \subset 3Q} | \mathscr{A}_S f | , |g| \right \rangle \leq \langle f \rangle_{3E,\infty} \langle |g| \rangle_{3E, 1} |E|.
\end{align*}
The restricted weak-type estimate is finally obtained by noting

\begin{align*}
& \left \langle \sum_{Q \in \mathcal{Q}(E)}  1_Q \sup_{S \not \subset 3Q} | \mathscr{A}_S f | , |g| \right \rangle \\ \leq & \min \{ \langle 1_{E_1} \rangle_{3E,\infty} \langle 1_{E_2} \rangle_{3E, 1} ,A  \langle 1_{E_1} \rangle_{3E,p_1} \langle 1_{E_2} \rangle_{3E, r_1} \} |E| \\ \leq & A \langle 1_{E_1} \rangle_{3E,p} \langle 1_{E_2} \rangle_{3E, r} |E|.
\end{align*}

\end{proof}
\section{Counterexamples}
We finish by showing the necessary statements at the ends of Theorems \ref{Thm:1} and \ref{Thm:2}. 
 \begin{prop}\label{P:Imp-N}
Let $d \geq 5$. A necessary condition for 
\begin{align*}
\left| \left| \sup_{\Lambda \leq \lambda < 2 \Lambda} | \mathscr{A}_\lambda  | \right| \right|_{\ell^p \to \ell^{r^\prime}} \leq A \Lambda^{d(1/r^\prime - 1/p)}
\end{align*}
to hold for all $\Lambda \in 2^{\mathbb{N}}$ is $\max\left\{ \frac{1}{p} +\frac{2}{d}, \frac{1}{r}+\frac{2}{pd} \right\} \leq 1$.

\end{prop}
\begin{proof}
The necessity of $\frac{1}{p} + \frac{2}{d} \leq 1$ follows by considering $f=1_{\{0\}}$. Indeed, it is straightforward to see that for $d \geq 5$ and this choice of $f$
\begin{align*}
\sup_{\Lambda \leq \lambda <2 \Lambda} f (n) \geq A \Lambda^{2-d} 1_{|n| \simeq \Lambda} (n),
\end{align*}
and the uniform estimate 
\begin{align*}
 \Lambda^{-d/r^\prime} \left| \left| \sup_{\Lambda \leq \lambda < 2 \Lambda} \left| \mathscr{A}_\lambda f \right| \right| \right|_{\ell^{r^\prime}(\mathbb{Z}^d)}  \leq A \Lambda^{-d/p} || f||_{\ell^p(\mathbb{Z}^d)}
\end{align*}
implies $\Lambda^{2-d} \leq A \Lambda^{-d/p}$. That the condition $\frac{1}{p} +\frac{2}{d} \leq 1$ must hold follows by taking $\Lambda \in 2^{\mathbb{N}}$ arbitrarily large. 

The necessity of $ \frac{1}{r}+\frac{2}{pd} \leq 1$ follows from setting $f_\Lambda=1_{|n|=\Lambda}(n)$. Then it is immediate that $\sup_{\Lambda \leq \lambda < 2 \Lambda} \mathscr{A}_\lambda f_\Lambda (0) = 1$ and 
\begin{align*}
\Lambda^{-d/r^\prime} \leq \Lambda^{-d/r^\prime} \left|\left| \sup_{\Lambda \leq \lambda <2 \Lambda}  \mathscr{A}_\lambda f_\Lambda  \right| \right|_{\ell^{r^\prime}(\mathbb{Z}^d)} \leq A \Lambda^{-d/p} || f_\Lambda||_{\ell^p(\mathbb{Z}^d)} \leq A \Lambda^{-d/p}  \Lambda^{(d-2)/p}. 
\end{align*}
The necessity of $\frac{1}{r} + \frac{2}{pd} \leq 1$ follows by again taking $\Lambda \in 2^{\mathbb{N}}$ arbitrarily large.

\end{proof}
Theorem \ref{Thm:1} follows from Propositions \ref{Cor:Imp-Diag} and \ref{P:Imp-N}.  Lastly, we record
\begin{prop}\label{P:Sp-N}
A necessary condition for 
\begin{align*}
 \left| \left| \sup_{\Lambda \leq \lambda < 2 \Lambda} \left| \mathscr{A}_\lambda  \right|: (p,r) \right| \right|  \leq A \Lambda^{d(1/r^\prime - 1/p)} 
\end{align*}
 to hold for all $\Lambda \in 2^{\mathbb{N}}$ is $\max\left\{ \frac{1}{p} +\frac{2}{d}, \frac{1}{r}+\frac{2}{pd} \right\} \leq 1$. 
\end{prop}
\begin{proof}
Fix $\Lambda \in 2^{\mathbb{N}}$ and set $f(n) = 1_{\{0\}}$ and $g_\Lambda(n)=1_{\Lambda \leq |n| <2 \Lambda} (n)$. Then observe
\begin{align*}
\left \langle  \sup_{\Lambda \leq \lambda <2 \Lambda}\mathscr{A}_{\lambda} f_, g_\Lambda \right \rangle \geq \frac{\Lambda^2}{A} . 
\end{align*}
However, $\sup_{\mathcal{S}} \Lambda_{\mathcal{S}, p,r}(f_\Lambda, g) \leq  A \Lambda^{d(1-1/p)}$. Indeed, for any sparse collection $\mathcal{S}$, 
\begin{align*}
\sum_{Q \in \mathcal{S}} \langle f \rangle_{Q,p} \langle g_\Lambda \rangle_{Q,r} |Q| =\sum_{\substack{ Q \in \mathcal{S}: \ell(Q) \geq \frac{\Lambda}{\sqrt{d}} \\  Q \ni \{0\}}} \langle f \rangle_{Q,p} \langle g_\Lambda \rangle_{Q,r} |Q| \leq A\Lambda^{d(1-1/p)}. 
\end{align*}
The claim that $\frac{1}{p} +\frac{2}{d} \leq  1$ then follows by taking $\Lambda \in 2^{\mathbb{N}} $ arbitrarily large. Moreover, for a given $\Lambda \in 2^{\mathbb{N}}$, set $f_\Lambda (n)=1_{|n|=\Lambda}(n)$ and $g(n)=1_{\{0\}}(n)$. Then   
\begin{align*}
\left \langle  \sup_{\Lambda \leq \lambda <2 \Lambda}\mathscr{A}_{ \lambda } f_\Lambda, g  \right \rangle \geq \frac{1}{A}. 
\end{align*}
However, $\sup_{\mathcal{S}} \Lambda_{\mathcal{S}, p,r}(f_\Lambda, g) \leq A \Lambda^{-2/p}\Lambda^{d(1-1/r)}$. The necessity of $\frac{1}{r}+ \frac{2}{pd} \leq 1$ then follows from taking $\Lambda \in 2^{\mathbb{N}}$ arbitrarily large. 
\end{proof}
Theorem \ref{Thm:2} follows from Theorem \ref{Thm:Sp-Corn} and Proposition \ref{P:Sp-N}.

\bibliography{sparse_discrete-spherical}	
\bibliographystyle{amsplain}	
\end{document}